\documentclass[a4paper]{amsart}

\usepackage{amsmath,amssymb,amsthm}
\usepackage{fullpage}
\usepackage[colorlinks,linkcolor=blue,urlcolor=blue,citecolor=blue,destlabel=true]{hyperref}
\usepackage{tikz-cd}
\usepackage{mathtools}
\usepackage{array} 
\usepackage{caption}
\usepackage{soul}

\theoremstyle{plain}
\newtheorem{lemma}{Lemma}[section]
\newtheorem{theorem}[lemma]{Theorem}
\newtheorem*{theorem-main}{Theorem~\ref{thm:main}}

\newtheorem{proposition}[lemma]{Proposition}

\theoremstyle{definition}
\newtheorem{definition}[lemma]{Definition}
\newtheorem{example}[lemma]{Example}
\newtheorem{nonexample}[lemma]{Non-Example}
\newtheorem{question}[lemma]{Question}
\newtheorem{remark}[lemma]{Remark}

\newcommand{\R}{\mathbb{R}}

\newcommand{\Z}{\mathbb{Z}}
\newcommand{\f}{\varphi}

\newcommand{\vr}[2]{\mathrm{VR}(#1;#2)}
\newcommand{\diam}{\mathrm{diam}}

\newcommand{\image}{\mathrm{im}}
\newcommand{\interior}{\mathrm{int}}

\newcommand{\conv}{\mathrm{Conv}}

\newcommand{\rank}{\mathrm{rank}}
\newcommand{\chull}{\mathrm{cHull}}
\newcommand{\locdiam}{\mathrm{localDiam}}

\newcommand{\nicea}{\mathcal{A}}

\newcommand{\rr}{\mathcal{R}}

\begin{document}

\title{Lower bounds on the homology of Vietoris--Rips complexes of hypercube graphs}
\author{Henry Adams}
\email{henry.adams@ufl.edu}
\author{\v{Z}iga Virk}
\email{ziga.virk@fri.uni-lj.si}
\date{\today}

\begin{abstract}
We provide novel lower bounds on the Betti numbers of Vietoris--Rips complexes of hypercube graphs of all dimensions, and at all scales.
In more detail, let $Q_n$ be the vertex set of $2^n$ vertices in the $n$-dimensional hypercube graph, equipped with the shortest path metric.
Let $\vr{Q_n}{r}$ be its Vietoris--Rips complex at scale parameter $r \ge 0$, which has $Q_n$ as its vertex set, and all subsets of diameter at most $r$ as its simplices.
For integers $r<r'$ the inclusion $\vr{Q_n}{r}\hookrightarrow\vr{Q_n}{r'}$ is nullhomotopic, meaning no persistent homology bars have length longer than one, and we therefore focus attention on the individual spaces $\vr{Q_n}{r}$.
We provide lower bounds on the ranks of homology groups of $\vr{Q_n}{r}$.
For example, using cross-polytopal generators, we prove that the rank of $H_{2^r-1}( \vr{Q_{n}}{r})$ is at least $2^{n-(r+1)}\binom{n}{r+1}$.
We also prove a version of \emph{homology propagation}: 
if $q\ge 1$ and if $p$ is the smallest integer for which $\rank H_q(\vr{Q_p}{r})\neq 0$, then $\rank H_{q}( \vr{Q_{n}}{r}) \ge \sum_{i=p}^n 2^{i-p} \binom{i-1}{p-1}
\cdot \rank H_{q}( \vr{Q_{p}}{r})
$
for all $n \ge p$.
When $r\le 3$, this result and variants thereof provide tight lower bounds on the rank of $H_{q}(\vr{Q_{n}}{r})$ for all $n$, and for each $r \ge 4$ we produce novel lower bounds on the ranks of homology groups.
Furthermore, we show that for each $r\ge 2$, the homology groups of $\vr{Q_{n}}{r}$ for $n \ge 2r+1$ contain propagated homology not induced by the initial cross-polytopal generators.
\end{abstract}



\maketitle


\section{Introduction}

Let $Q_n$ be the vertex set of the hypercube graph, equipped with the shortest path metric.
In other words, $Q_n$ can be thought of the set of all $2^n$ binary strings of $0$'s and $1$'s equipped with the Hamming distance, or alternatively, as the set $\{0,1\}^n\subseteq \R^n$ equipped with the $\ell^1$ metric.

In this paper, we study the topology of the \emph{Vietoris--Rips simplicial complexes} of $Q_n$.
Given a metric space $X$ and a scale $r\ge 0$, the Vietoris--Rips simplicial complex $\vr{X}{r}$ has $X$ as its vertex set, and a finite subset $\sigma \subseteq X$ as a simplex if and only if the diameter of $\sigma$ is at most $r$.
Originally introduced for use in algebraic topology~\cite{Vietoris27} and geometric group theory~\cite{bridson2011metric,Gromov},
Vietoris--Rips complexes are now a commonly-used tool in applied and computational topology in order to approximate the shape of a dataset~\cite{Carlsson2009,EdelsbrunnerHarer}.
Important results include the fact that nearby metric spaces give nearby Vietoris--Rips persistent homology barcodes~\cite{ChazalDeSilvaOudot2014,chazal2009gromov}, that Vietoris--Rips complexes can be used to recover the homotopy types of manifolds~\cite{Hausmann1995,Latschev2001,virk2021rips,majhi2023demystifying}, and that Vietoris--Rips persistent homology barcodes can be efficiently computed~\cite{bauer2021ripser}.
Nevertheless, not much is known about Vietoris--Rips complexes of manifolds or of simple graphs at large scale parameters, unless the manifold is the circle~\cite{AA-VRS1}, unless the graph is a cycle graph~\cite{Adamaszek2013,AAFPP-J}, or unless one restricts attention to $1$-dimensional homology~\cite{virk20201,gasparovic2018complete}.

Let $\vr{Q_n}{r}$ be the Vietoris--Rips complex of the vertex set of the $n$-dimensional hypercube at scale parameter $r$.
The homotopy types of $\vr{Q_n}{r}$ are known for $r\le 3$ (and otherwise mostly unknown); see Table~\ref{table:homotopy-types}.
For $r=0$, $\vr{Q_n}{0}$ is the disjoint union of $2^n$ vertices, and hence homotopy equivalent to a wedge sum $(2^n-1)$-fold wedge sum of zero-dimensional spheres.
For $r=1$, $\vr{Q_n}{1}$ is a connected graph (the hypercube graph), which by a simple Euler characteristic computation is homotopy equivalent to a $((n-2)2^{n-1}+1)$-fold wedge sum of circles.
For $r=2$, Adams and Adamaszek~\cite{adamaszek2022vietoris} prove that $\vr{Q_n}{2}$ is homotopy equivalent to a wedge sum of $3$-dimensional spheres; see Theorem~\ref{ThmAA} for a precise statement which also counts the number of $3$-spheres.
For $r=3$, Shukla proved in~\cite[Theorem~A]{shukla2022vietoris} that for $n\ge 5$, the $q$-dimensional homology of $\vr{Q_n}{3}$ is nontrivial if and only if $q=7$ or $q=4$.
The study of $r=3$ was furthered by Feng~\cite{feng2023homotopy} (based on work by Feng and Nukula~\cite{feng2023vietoris}), who proved that $\vr{Q_n}{3}$ is always homotopy equivalent to a wedge sum of $7$-spheres and $4$-spheres; see Theorem~\ref{ThmZiqin} for a precise statement which also counts the number of spheres of each dimension.
When $r=n-1$, $\vr{Q_n}{n-1}$ is isomorphic to the boundary of the cross-polytope with $2^n$ vertices, and hence is homeomorphic to a sphere of dimension $2^{n-1}-1$.
For $r\ge n$, the space space $\vr{Q_n}{n}$ is a complete simplex, and hence contractible.
However, there is an entire infinite ``triangle'' of parameters, namely $r\ge 4$ and $r\le n-2$, for which essentially nothing is known about the homotopy types of $\vr{Q_n}{r}$.

\begin{table}[htb]
{\footnotesize
\def\arraystretch{1.2}
\begin{tabular}{| >{$} c <{$} | >{$} c <{$} | >{$} c <{$} | >{$} c <{$} | >{$} c <{$} | >{$} c <{$} | >{$} c <{$} | >{$} c <{$} | >{$} c <{$} |
}
\hline
_n \backslash ^r & 0 & 1 & 2 & 3 & 4 & 5 & 6 & 7 \\
\hline
1 & S^0 & * & * & * & * & * & * & * \\
\hline
2 & \bigvee^3 S^0 & S^1 & * & * & * & * & * & * \\
\hline
3 & \bigvee^7 S^0 & \bigvee^5 S^1 & S^3 & * & * & * & * & * \\
\hline
4 & \bigvee^{15} S^0 & \bigvee^{17} S^1 & \bigvee^9 S^3 & S^7 & * & * & * & * \\
\hline
5 & \bigvee^{31} S^0 & \bigvee^{49} S^1 & \bigvee^{49} S^3 & \bigvee^{10} S^7 \vee S^4 & S^{15} & * & * & * \\
\hline
6 & \bigvee^{63} S^0 & \bigvee^{129} S^1 & \bigvee^{209} S^3 & \bigvee^{60} S^7 \vee \bigvee^{11} S^4 & \textcolor{blue}{\beta_{15}\ge 12} & S^{31} & * & * \\
\hline
7 & \bigvee^{127} S^0 & \bigvee^{321} S^1 & \bigvee^{769} S^3 & \bigvee^{280} S^7 \vee \bigvee^{71} S^4 & \textcolor{blue}{\beta_{15}\ge 84} & \textcolor{blue}{\beta_{31}\ge 14} & S^{63} & * \\
\hline
8 & \bigvee^{255} S^0 & \bigvee^{769} S^1 & \bigvee^{2561} S^3 & \bigvee^{1120} S^7 \vee \bigvee^{351} S^4 & \textcolor{blue}{\beta_{15}\ge 448} & \textcolor{blue}{\beta_{31}\ge 112} & \textcolor{blue}{\beta_{63}\ge 16} & S^{127} \\
\hline
9 & \bigvee^{511} S^0 & \bigvee^{1793} S^1 & \bigvee^{7937} S^3 & \bigvee^{4032} S^7 \vee \bigvee^{1471} S^4 & \textcolor{blue}{\beta_{15}\ge 2016} & \textcolor{blue}{\beta_{31}\ge 672} & \textcolor{blue}{\beta_{63}\ge 144} & \textcolor{blue}{\beta_{127}\ge 18} \\
\hline
10 & \bigvee^{1023} S^0 & \bigvee^{4097} S^1 & \bigvee^{23297} S^3 & \bigvee^{13440} S^7 \vee \bigvee^{5503} S^4 & \textcolor{blue}{\beta_{15}\ge 8064} & \textcolor{blue}{\beta_{31}\ge 3360} & \textcolor{blue}{\beta_{63}\ge 960} & \textcolor{blue}{\beta_{127}\ge 180} \\
\hline
11 & \bigvee^{2047} S^0 & \bigvee^{9217} S^1 & \bigvee^{65537} S^3 & \bigvee^{42240} S^7 \vee \bigvee^{18943} S^4 & \textcolor{blue}{\beta_{15}\ge 29568} & \textcolor{blue}{\beta_{31}\ge 14784} & \textcolor{blue}{\beta_{63}\ge 5280} & \textcolor{blue}{\beta_{127}\ge 1320} \\
\hline
12 & \bigvee^{4095
} S^0 & \bigvee^{20481} S^1 & \bigvee^{178177} S^3 & \bigvee^{126720} S^7 \vee \bigvee^{61183} S^4 & \textcolor{blue}{\beta_{15}\ge 101376} & \textcolor{blue}{\beta_{31}\ge 59136} & \textcolor{blue}{\beta_{63}\ge 25344} & \textcolor{blue}{\beta_{127}\ge 7920} \\
\hline
\end{tabular}
}
\caption{Black entries are the known homotopy types of $\vr{Q_n}{r}$; blue entries are sample novel lower bounds on the Betti numbers of $\vr{Q_n}{r}$ based on Theorem~\ref{ThmMain1}.
(See Table~\ref{table:homotopy-types2} for improved lower bounds for the $r=4$ column.)}
\label{table:homotopy-types}
\end{table}

In this paper, instead of focusing on a single value of $r$, we provide novel lower bounds on the ranks of homology groups of $\vr{Q_n}{r}$ for all values of $r$.
Some of these lower bounds are shown in blue in Table~\ref{table:homotopy-types}: 
using cross-polytopal generators, in Theorem~\ref{ThmMain1} we prove $\rank H_{2^r-1}( \vr{Q_{n}}{r}) \geq 2^{n-(r+1)}\binom{n}{r+1}$ (although we show in Section \ref{sec:geometric} that for $r\ge 2$ and $n > 2r$ this does not constitute the entire reduced homology in all dimensions).
This is the first result showing that the topology of $\vr{Q_n}{r}$ is nontrivial for \emph{all} values of $r\le n-1$.
Furthermore, we often show that $\vr{Q_n}{r}$ is far from being contractible, with the rank of $(2^r-1)$-dimensional homology tending to infinity exponentially fast as a function of $n$ (with $n$ increasing and with $r$ fixed). 

Our general strategy, which we refer to as \emph{homology propagation}, is as follows.
Let $q\ge 1$.
Suppose that one can show that the $q$-dimensional homology group $H_q(\vr{Q_{p}}{r})$ is nonzero (for example, using a homology computation on a computer, or alternatively a theoretical result such as the mentioned cross-polytopal elements or the geometric generators of Section \ref{sec:geometric}).
Then we provide lower bounds on the ranks of the homology groups $H_q(\vr{Q_{n}}{r})$ for all $n\ge p$.
In particular, in Theorem~\ref{ThmMain4} we prove that if $p\ge 1$ is the smallest integer for which $H_q(\vr{Q_{p}}{r})\neq 0$, then
\[\rank H_{q}(\vr{Q_{p}}{r}) \geq \sum_{i=p}^n 2^{i-p} \binom{i-1}{p-1} \cdot \rank H_{q}(\vr{Q_{p}}{r}).\]
Thus, a homology computation for a low-dimensional hypercube $Q_p$ has consequences for the homology of $\vr{Q_{n}}{r}$ for all $n\ge p$.
See Table~\ref{table:homotopy-types2} for some consequences of this result and of related results.

As we explain in Section~\ref{SubsComparison}, when $r\le 3$ our results are known to provide tight lower bounds on all Betti numbers of $\vr{Q_{n}}{r}$.
We take this as partial evidence that our novel results on the Betti numbers of $\vr{Q_{n}}{r}$ for $r\ge 4$ are likely to be good lower bounds, though we do not know how close they are to being tight as no upper bounds are known.
Indeed, the main ``upper bound'' we know of on the Betti numbers of $\vr{Q_{n}}{r}$ is a triviality result for 2-dimensional homology: Carlsson and Filippenko~\cite{carlsson2020persistent} prove that $H_2(\vr{Q_n}{r})=0$ for all $n$ and $r$.

For integers $r<r'$, we prove via a simple argument that the inclusion $\vr{Q_n}{r}\hookrightarrow\vr{Q_n}{r'}$ is nullhomotopic.
Therefore, there are no persistent homology bars of length longer than one, and all homological information about the filtration $\vr{Q_n}{\bullet}$ is determined by $\vr{Q_n}{r}$ for individual integer values of $r$.

Though we have stated our results for $Q_n=\{0,1\}^n$ equipped with the $\ell^1$ metric, we remark that these results hold for any $\ell^p$ metric with $1\le p < \infty$.
Indeed for $x,y\in Q_n$, the $i$-th coordinates of $x$ and $y$ differ by either $0$ or $1$ for each $1\le i\le n$, and hence $\vr{(Q_n,\ell^p)}{r}=\vr{(Q_n,\ell^1)}{r^p}$.
So, our results can be translated into any $\ell^p$ metric by a simple reparametrization of scale.

We expect that some of our work could be transferred over to provide results for \v{C}ech complexes of hypercube graphs, as studied in~\cite{adams2022v}, though we do not pursue that direction here.

We begin with some preliminaries in Section~\ref{sec:preliminaries}.
In Section~\ref{sec:contractions-ph} we review contractions, and we prove that $\vr{Q_n}{\bullet}$ has no persistent homology bars of length longer than one.
In Section~\ref{sec:maximal} we use cross-polytopal generators to prove $\rank H_{2^r-1}( \vr{Q_{n}}{r}) \geq 2^{n-(r+1)}\binom{n}{r+1}$.
We introduce concentrations in Section~\ref{sec:concentrations}, which we use to prove our more general forms of homology propagation in Section~\ref{sec:contractions-homology}.
In Section~\ref{sec:geometric} we prove the existence of novel lower-dimensional homology generators, and we conclude with some open questions in Section~\ref{sec:conclusion}.

\section{Preliminaries and geometry of hypercubes.}
\label{sec:preliminaries}

\subsection{Homology}

All homology groups will be considered with coefficients in $\mathbb{Z}$ or in a field.
The rank of a finitely generated abelian group is the cardinality of a maximal linearly independent subset.
We let $\beta_q$ denote the $q$-th \emph{Betti number} of a space, i.e., the rank of the $q$-dimensional homology group.

\subsection{Hypercubes}
Hypercubes are among the simplest examples of product spaces.

\begin{definition}
Given $n\in \{1,2,\ldots\}$, the \emph{hypercube graph} $Q_n$ is the metric space $\{0,1\}^n$, equipped with the $\ell^1$ metric.
In particular, the elements of the space are $n$-tuples $(a_1, a_2, \ldots, a_n)=(a_i)_{i\in [n]}$ with $a_i\in \{0,1\}$, and the $\ell^1$ distance is defined as
\[
d \big((a_1, a_2, \ldots, a_n),(b_1, b_2, \ldots, b_n)\big)= 
\sum_{i=1}^n |a_i-b_i|.
\]
\end{definition}
\noindent In other words, the distance between two $n$-tuples is the number of coordinates in which they differ.

For $a\in Q_n$, its antipodal point $\bar a$ is given as $\bar a = (1,1,\ldots, 1)-a$.
In particular, $\bar a$ is the furthest point in $Q_n$ from $a$ and thus shares no coordinate with $a$.
Observe that $d(a,\bar a)=n$.

\subsection{Vietoris--Rips complexes}
A Vietoris--Rips complex is a way to ``thicken'' a metric space, as we describe via the definitions below.

\begin{definition}
Given a metric space $X$ and a finite subset $A \subseteq X$, the \emph{diameter of $A$} is 
\[
\diam(A) = \max_{a,b\in A} d(a,b).
\]
The \emph{local diameter} of $A$ at a point $a\in A$ equals
\[
\locdiam(A,a) = \max_{b\in A} d(a,b).
\]
\end{definition}

\begin{definition}
Given $r \geq 0$ and a metric space $X$ the \emph{Vietoris--Rips complex} $\vr{X}{r}$ is the simplicial complex with vertex set $X$, and with a finite subset $\sigma \subseteq X$ being a simplex whenever $\diam (\sigma) \leq r$.
\end{definition}

This is the \emph{closed} Vietoris--Rips complex, since we are using the convention $\le$ instead of $<$.
But, since the metric spaces $Q_n$ are finite, all of our results have analogues if one instead considers the \emph{open} Vietoris--Rips complex that uses the $<$ convention.

In~\cite{adamaszek2022vietoris} it was proven that $\vr{Q_n}{2}$ is homotopy equivalent to a wedge sum of $3$-dimensional spheres:

\begin{theorem}[Theorem~1 of~\cite{adamaszek2022vietoris}]
\label{ThmAA} 
For $n\ge 3$, we have the homotopy equivalence
\[\vr{Q_n}{2} \simeq \bigvee_{c_n} S^3, \textrm{ where } c_n= \sum_{0 \leq j < i < n}(j+1)(2^{n-2}- 2^{i-1}).\]
\end{theorem}
\noindent See~\cite{saleh2023vietoris} for some relationships between this result and generating functions.

In~\cite{feng2023homotopy}, it was proven that $\vr{Q_n}{3}$ is always homotopy equivalent to a wedge sum of $7$-spheres and $4$-spheres:

\begin{theorem}[Theorem~24 of~\cite{feng2023homotopy}]
\label{ThmZiqin}
For $n\ge 5$, we have the homotopy equivalence
\[\vr{Q_n}{3} \simeq \bigvee_{2^{n-4}\binom{n}{4}} S^7 \ \vee \bigvee_{\sum_{i=4}^{n-1} 2^{i-4} \binom{i}{4}} S^4.\]
\end{theorem}

\subsection{Embeddings of hypercubes}
\label{ssec:embeddings}

For $k$ a positive integer, let $[k]=\{1,2,\ldots, k\}$.
Given $p \in [n-1]$ there are many isometric copies of $Q_p$ in $Q_n$.
For any subset $S \subseteq [n]$ of cardinality $p$ we can isometrically embed $Q_p$ in $Q_n$, using set $S$ as its variable coordinates, and leaving the rest of the entries fixed.
In more detail, we define an isometric embedding $\iota_S^b \colon Q_p \hookrightarrow Q_n$ associated to a subset $S=\{s_1,s_2, \ldots, s_p\} \subseteq [n]$ of coordinates and an offset $(b_i)_{i\in [n]\setminus S} \in \{0,1\}^{n-|S|}$, maps $(a_i)_{i\in [p]}$ to $(a'_i)_{i\in [n]}$ with:
\begin{itemize}
 \item $a'_{s_i}=a_i$ for $i\in [p]$, and
 \item $a'_i=b_i$ otherwise.
\end{itemize}
Given a fixed set $S$, there are $2^{n-p}$ such embeddings $\iota_S$, each associated to a different offset $b$.
Let $\pi_S \colon Q_n \to Q_p$ be the map projecting onto the coordinates in $S$.
Then $\pi_s \circ \iota_S = id_{Q_p}$ for any map $\iota_S$ (i.e., for any choice of an offset $b$).
Given an offset $(b_i)_{i\in [n]\setminus S}$, let $Q_p^b$ denote the image of $\iota_S^b$ corresponding to the offset $b$, and let $\pi_S^b \colon Q_n \to Q_p^b$ be defined as $\iota_S^b \circ \pi_S$.
Given $B \subseteq Q_n$ its \emph{cubic hull} $\chull(B)$ is the smallest isometric copy of a cube (i.e., the image of $Q_{p'}$ via some map $\iota$) containing $B$.

For our purposes we will only consider isometric embeddings $Q_p \hookrightarrow Q_n$ (also denoted by $Q_p \leq Q_n$) that retain the order of coordinates, although any permutation of coordinates of $Q_p$ results in a non-constant isometry of $Q_p$ and thus a different isometric embedding into $Q_n$.
With this convention of retaining the coordinate order, there are $\binom{n}{p} 2^{n-p}$ isometric embeddings 
$\iota \colon Q_p \hookrightarrow Q_n$ and $\binom{n}{p}$ projections $\pi \colon Q_n \to Q_p$.

\section{Contractions and the persistent homology of hypercubes}
\label{sec:contractions-ph}

In this section we prove the following results.
First, fix the scale $r\ge 0$, and let $p\le n$.
An isometric embedding $Q_p \hookrightarrow Q_n$ gives an inclusion $\vr{Q_p}{r}\hookrightarrow \vr{Q_n}{r}$, which is injective on homology in all dimensions.
Alternatively, fix dimension $n$, and consider integer scale parameters $r < r'$.
The inclusion $\vr{Q_n}{r}\hookrightarrow \vr{Q_n}{r'}$ is nullhomotopic, and hence the filtration $\vr{Q_n}{\bullet}$ has no persistent homology bars of length longer than one.
These results follow from the properties of contractions, which we introduce now.


\subsection{Contractions}
\label{ssec:contractions}

A map $f \colon X \to A$ from a metric space $(X,d)$ onto a closed subspace $A \subseteq X$ is a \emph{contraction} if $f|_{A}=id_A$ and if $d(f(x),f(y)) \leq d(x,y)$ for all $x,y\in X$.

Our interest in contractions stems from the fact that if a contraction $X \to A$ exists, then the homology of a Vietoris--Rips complex of $A$ maps injectively into the homology of the corresponding Vietoris--Rips complex of $X$:

\begin{proposition}[\cite{virk2022contractions}]
\label{PropContrEmbed}
 If $f\colon X \to A$ is a contraction, then the embedding $A \hookrightarrow X$ induces injections on homology $H_q(\vr{A}{r})\to H_q(\vr{X}{r})$ for all integers $q\ge 0$ and scales $r\ge 0$.
\end{proposition}

We prove that the projections from a higher-dimensional cube to a lower-dimensional cube in Section~\ref{ssec:embeddings} are contractions:

\begin{lemma}
\label{Lemma1}
Given fixed $p\in[n-1]$, set $S\subseteq [n]$ of cardinality $p$, and offset $b$ as in Section~\ref{ssec:embeddings}, the following hold:
\begin{enumerate}
\item Maps $\pi_S$ and $\pi_S^b$ are contractions.
\item For each $x\in Q_p^b$ and $y\in Q_n$, we have 
$
d(x,y) = d(x, \pi_S^b(y)) + d(\pi_S^b(y),y)
$.
\item For each offset $b'$:
\begin{enumerate}
	\item For each $x,y\in Q_p^{b'}$, we have $d(x,y)=d(\pi_p^b(x),\pi_p^b(y))$.
	\item For each $x\in Q_p^{b'}$ and $y\notin Q_p^{b'}$, we have $d(x,y) -1 \geq d(\pi_p^b(x),\pi_p^b(y)) \geq d(x,y) - (n-p)$.
	\end{enumerate}
\end{enumerate}
\end{lemma}

\begin{proof}
For $x,y\in Q_n$, the distance $d(x,y)$ is the number of components in which $x$ and $y$ differ.
On the other hand, $d(\pi_S^b(x), \pi_S^b(y))$ is the number of components from $S$ in which $x$ and $y$ differ (and the same holds for map $\pi_S$ instead of $\pi_S^b$ as well).
Thus $d(x,y) \leq d(\pi_S^b(x), \pi_S^b(y))$ with:
\begin{itemize}
 \item $d(x,y) = d(\pi_S^b(x), \pi_S^b(y))$ if $x,y\in Q_p^{b'}$ as their coordinates outside $S$ agree, yielding (1) and (3)(a), and
 \item (3)(b) follows from the fact that for each $x\in Q_p^{b'}, y\notin Q_p^{b'}$, the number of coordinates outside of $S$ on which $x$ and $y$ disagree is at least $1$ (due to $x,y$ not being in the same $Q_p^*$) and at most $n-p$ (which is the cardinality of $[n]\setminus S$).
\end{itemize}
Item (2) follows from the observation that:
\begin{itemize}
 \item $d(x, \pi_S^b(y))$ is the number of components  from $S$ in which $x$ and $y$ differ, and
 \item $d(\pi_S^b(y),y)$ is the number of components  from $[n]\setminus S$ in which $x$ and $y$ differ, as $x\in Q_p^b$.
\end{itemize}
\end{proof}

Since each of the projections $\pi \colon Q_n \to Q_p$ is a contraction by Lemma~\ref{Lemma1}, Proposition~\ref{PropContrEmbed} then implies that each of the embeddings $Q_p \hookrightarrow Q_n$ induces an injective map on homology $H_q(\vr{Q_p}{r}) \to H_q(\vr{Q_n}{r})$ for all dimensions $q$.

\subsection{Persistent homology of hypercubes}
\label{ssec:ph}

The emphasis in modern topology is often on persistent homology arising from the Vietoris--Rips filtration.
However, in the setting of Vietoris--Rips complexes of hypercubes, persistent homology does not provide any more information beyond the homology groups at fixed scale parameters.
Indeed, the following proposition implies that for any integers $r < r'$, the inclusion $\vr{Q_n}{r} \hookrightarrow \vr{Q_n}{r+1}$ induces a map that is trivial on homology.

\begin{proposition}
For any positive integers $n$ and $r$, the natural inclusion $\vr{Q_n}{r} \hookrightarrow \vr{Q_n}{r+1}$ is homotopically trivial.
\end{proposition}

\begin{proof}
We first claim that the inclusion $\vr{Q_n}{r} \hookrightarrow \vr{Q_n}{r+1}$ is homotopic to the projection $\pi_{[n-1]}\colon \vr{Q_n}{r} \to \vr{Q_{n-1}}{r}$ in $ \vr{Q_n}{r+1}$.
In order to prove the claim we will show that the two maps are contiguous in $\vr{Q_n}{r+1}$ (i.e., for each simplex $\sigma \in \vr{Q_n}{r}$ the union $\sigma \cup \pi_{[n-1]}(\sigma)$ is contained in a simplex of $\vr{Q_n}{r+1}$), which implies that the two maps are homotopic.

Let $\sigma \in \vr{Q_n}{r}$.
By definition $\diam(\sigma)\leq r$.
As $\pi_{[n-1]}(\sigma)$ is obtained by dropping the final coordinate we also have $\diam(\pi_{[n-1]}(\sigma))\leq r$.
Taking $x\in \sigma$ and $y\in \pi_{[n-1]}(\sigma)$, i.e.\ $y=\pi_{[n-1]}(y')$ for some $y'\in\sigma$, we see that
\[
d(x,y)\leq d(x,y')+d(y',y) \leq r + 1
\]
as $d(y,y')\leq 1$.
This $\sigma \cup \pi_{[n-1]}(\sigma) \in \vr{Q_n}{r+1}$, and the claim is proved.

We proceed inductively, proving that each projection $\pi_{[k]}\colon \vr{Q_n}{r} \to \vr{Q_k}{r}$ is homotopic to the projection $\pi_{[k-1]}\colon \vr{Q_n}{r} \to \vr{Q_{k-1}}{r}$ in $\vr{Q_n}{r+1}$, by the same argument as above.
As a result, the embedding $\vr{Q_n}{r} \hookrightarrow \vr{Q_n}{r+1}$ is homotopic to the projection $\pi_{\{1\}}\colon \vr{Q_n}{r} \to \vr{Q_1}{r}$.
Since $\vr{Q_1}{r}$ is clearly contractible, this completes the proof.
\end{proof}

\section{Homology bounds via cross-polytopes and maximal simplices}
\label{sec:maximal}

Fix a scale $r\ge 2$, and consider an isometric embedding $\iota \colon Q_{r+1} \hookrightarrow Q_n$ for $n\ge r+1$.
The aim of this section is to prove not only that the induced map $\vr{Q_{r+1}}{r}\hookrightarrow \vr{Q_n;r}$ is injective on $(2^r-1)$-dimensional homology, but also that different (ordered) embeddings $\iota$ produce independent homology generators.
Let us explain this in detail.

We first observe that $\vr{Q_{r+1}}{r}$ is homeorphic to a $(2^r-1)$-dimensional sphere, i.e.\ $\vr{Q_{r+1}}{r} \cong S^{2^r-1}$.
The reason is that each vertex $x \in Q_{r+1}$ is connected by an edge in $\vr{Q_{r+1}}{r}$ to every vertex of $Q_{r+1}$ \emph{except} for $\bar x$, the antipodal vertex.
Therefore, after taking the clique complex of this set of edges, we see that $\vr{Q_{r+1}}{r}$ is isomorphic (as simplicial complexes) to the boundary of the cross-polytope with $2^{r+1}$ vertices.
This cross-polytope is a $2^r$-dimensional ball in $2^r$-dimensional Euclidean space, and therefore its boundary is a sphere of dimension $2^r-1$.
In particular, $\rank H_{2^r-1}( \vr{Q_{r+1}}{r})=1$.

Since $\vr{Q_{r+1}}{r}$ is the boundary of a cross-polytope, there is a convenient $(2^r -1)$-dimensional cycle $\gamma$ generating $H_{2^r-1}( \vr{Q_{r+1}}{r})$.
Define the set of maximal antipode-free simplices as 
\[
\nicea_r = \{Y \subseteq Q_{r+1}  \mid x \in Y \Leftrightarrow \bar x \notin Y\}.
\]
The cycle $\gamma$ is defined as the sum of appropriately oriented elements of $\nicea_r$.
The space $Q_{r+1}$ consists of $2^{r+1}$ points, which can be partitioned into $2^r$ pairs of mutually antipodal points.
If a subset of $Q_{r+1}$ contains exactly one point from each such pair, it is of cardinality $2^r$.
Thus $\nicea_r$ consists of sets of cardinality $2^r$.
Given $x\in Q_{r+1}$, the only element of $Q_{r+1}$ which disagrees with $x$ on all $r+1$ coordinates is $\bar x$.
As a result each element of $\nicea_r$ is  of diameter at most $r$ and thus a simplex of $ \vr{Q_{r+1}}{r}$.
Observe also that any element of $\nicea_r$ is a maximal simplex of $ \vr{Q_{r+1}}{r}$: adding any point to such a simplex would mean the presence of an antipodal pair, and so the diameter would thus grow to $r+1$.

As explained above, the embeddings $\iota \colon Q_{r+1} \hookrightarrow Q_n$ induce injections on homology by Lemma~\ref{Lemma1} and Proposition~\ref{PropContrEmbed}.
The fact that these embeddings give independent homology generators is formalized in the following statement, which is also the main result of this section.

\begin{theorem}
\label{ThmMain1}
For $r \geq 2$, 
\[
\rank H_{2^r-1}( \vr{Q_{n}}{r}) \geq 2^{n-(r+1)}\binom{n}{r+1}.
\]
\end{theorem}

The proof will be provided at the conclusion of the section.
Recall that $2^{n-(r+1)}\binom{n}{r+1}$ is the number of different (ordered) embeddings $\iota \colon Q_{r+1} \hookrightarrow Q_n$.
We will use maximal simplices and the pairing between homology and cohomology in order to prove that these $2^{n-(r+1)}\binom{n}{r+1}$ different embeddings provide independent cross-polytopal generators for homology.

\begin{proposition}
\label{PropCap}
Suppose $K$ is a simplicial complex and $\sigma$ is a maximal simplex of dimension $p$ in $K$.
If there is a $p$-cycle $\alpha$ in $K$ in which $\sigma$ appears with a non-trivial coefficient $\lambda$, then any representative $p$-cycle of $[\alpha]$ also contains $\sigma$ with the same coefficient $\lambda$.
\end{proposition}

\begin{proof}
As $\sigma$ is maximal, the $p$-cochain mapping $\sigma$ to $1$ and all other $p$-simplices to $0$ is a $p$-cocycle denoted by $\omega_\sigma$.
Utilizing the cap product we see that for each representative $\alpha'$ of $[\alpha]$, the cap product $[\omega_\sigma] \frown [\alpha'] = \lambda$ is the coefficient of $\sigma$ in $\alpha'$.
\end{proof}

\begin{remark}
Proposition~\ref{PropCap} could also be proved directly.
If $\alpha$ and $\alpha'$ are homologous $p$-cycles, then their difference is a boundary of a $(p+1)$-chain.
The later cannot contain $\sigma$ since $\sigma$ is maximal; hence the coefficients of $\sigma$ in $\alpha$ and in $\alpha'$ coincide.
 
We emphasize the cohomological proof because we will use the cochain $\omega_\sigma$ again.
\end{remark}

We next focus on the construction of maximal simplices of $ \vr{Q_{r+1}}{r}$ which are furthermore also maximal simplices in $\vr{Q_{n}}{r}$.
The following is a simple criterion identifying such a simplex as a maximal simplex in $\vr{Q_{n}}{r}$; see Figure~\ref{fig:maximalSimplices}.
(We recall that the \emph{local diameter} of $\sigma\subseteq Q_n$ at a point $w\in \sigma$ is defined as $\locdiam(\sigma,w) = \max_{z\in \sigma} d(w,z)$.)

\begin{figure}[htb]
\centering
\includegraphics[width=3in]{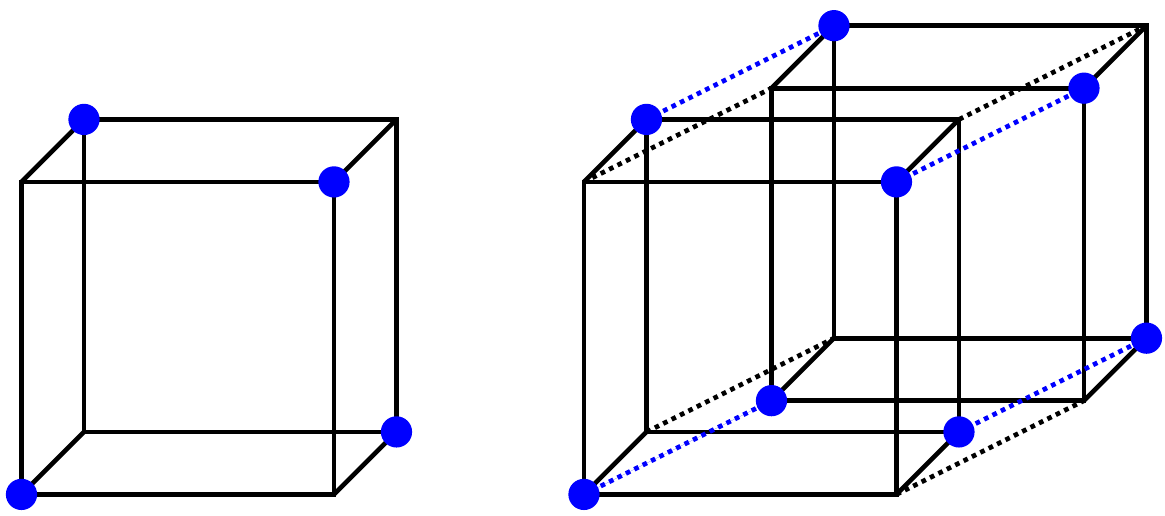}
\caption{\emph{(Left)} Subcube $Q_3$ with a maximal simplex $\sigma\in\vr{Q_3}{2}$ drawn in blue, illustrating Proposition~\ref{PropReduction}, and also Lemma~\ref{LemmaREvenAndOdd} when $r$ is even.
An inclusion of $\sigma$ in $Q_4$ also gives a maximal simplex $\iota_S^b(\sigma)\in\vr{Q_4}{2}$.
\emph{(Right)} Subcube $Q_4$ with a maximal simplex $\sigma\times\{0,1\}\in\vr{Q_4}{3}$ drawn in blue, illustrating Lemma~\ref{LemmaREvenAndOdd} when $r$ is odd.
}
\label{fig:maximalSimplices}
\end{figure}

\begin{proposition}
\label{PropReduction}
Let $n\ge r+1$, let $S \subseteq [n]$, and let $b$ be an associated offset.
Let $\sigma \subseteq Q_{r+1}^b$, and suppose $\sigma \in \nicea_r$ as a subset of $Q_{r+1}$.
If $\locdiam(\sigma,w) = r$ for all $w \in \sigma$, then $\sigma$ is a maximal simplex in $\vr{Q_{n}}{r}$.
\end{proposition}

\begin{proof}
Assume a point $x\in Q_n\setminus \sigma$ is added to $\sigma$.
We will show that this increases the diameter of $\sigma$ beyond $r$, by repeatedly using Lemma~\ref{Lemma1}.
\begin{itemize}
\item If $\pi_S^b(x) \notin \sigma$ then $\overline{\pi_S^b(x)} \in \sigma$ and thus 
\[
d(x,\overline{\pi_S^b(x)} ) = d(x,\pi_S^b(x)) + d(\pi_S^b(x),\overline{\pi_S^b(x)}) \geq 0 + (r+1) =r+1.
\]
\item If $\pi_S^b(x) \in \sigma$ then $d(x,\pi_S^b(x)) \geq 1$, and also the local diameter assumption implies there exists $y\in \sigma$ with $d(\pi_S^b(x),y)=r$.
Thus 
\[
d(x,y ) = d(x,\pi_S^b(x)) + d(\pi_S^b(x),y) \geq 1 + r.
\]
\end{itemize}
Hence $\sigma$ is maximal in $\vr{Q_n}{r}$.
\end{proof}

We now construct maximal simplices $\sigma$ in $\vr{Q_{r+1}}{r}$ that, by Proposition~\ref{PropReduction}, will remain maximal in $\vr{Q
_n}{r}$.
We recall that the \emph{cubic hull} $\chull(\sigma)$ is the smallest isometric copy of a cube containing $\sigma$.
That the convex hull of $\sigma$ is all of $Q_{r+1}$ will later be used to give the independence of homology generators in the proof of Theorem~\ref{ThmMain1}.

\begin{lemma}
\label{LemmaREvenAndOdd}
If $r \geq  2$, then there exists a maximal simplex $\sigma \subseteq Q_{r+1}$ from $\nicea_r$ with $\locdiam(\sigma,y)=r$ for all $y \in \sigma$, and with $\chull(\sigma)=Q_{r+1}$.
\end{lemma}

\begin{proof}
We first prove the case when $r\ge 2$ is odd (before afterwards handling the case when $r\ge 3$ is odd).
Define $\sigma$ as the collection of vertices in $Q_{r+1}$ whose coordinates contain an even number of values $1$.
As $r+1$ is odd, this means $x\in \sigma$ iff $\bar x \notin \sigma$, so $\sigma \in \nicea_r$.

We proceed by determining the local diameter.
Let $y\in \sigma$ and define $y'$ by taking $\bar y$ and flipping one of its coordinates.
Then $y' \in \sigma$ as it has an even number of ones, and it disagrees with $y$ on all coordinates except the flipped one, hence $d(y,y')=r$.
So $\locdiam(\sigma,y) = r$ for all $y\in\sigma$.

It remains to show that $\chull(\sigma)=Q_{r+1}$.
If $\chull(\sigma) \subsetneq Q_{r+1}$, there would be a single coordinate shared by all the points of $\sigma$.
However, as $r \geq 2$ we can prescribe any single coordinate as we please, and then fill in the rest of the coordinates to obtain a vertex of $\sigma$:
\begin{itemize}
\item  if the chosen coordinate was $1$, fill another coordinate as $1$ and the rest as $0$; 
\item  if the chosen coordinate was $0$, fill all other coordinates as $0$.
\end{itemize}

Next, we handle the case when $r\ge 3$ is odd.
Let $\tau$ be the the maximal simplex in $Q_r$ obtained in the proof of the even case.
Define 
\[
\sigma = \tau \times \{0,1\}\subseteq Q_{r+1}.
\]
Formally speaking, $\sigma = \iota_{[r]}^{(0)}(Q_r) \cup \iota_{[r]}^{(1)}(Q_r) = Q_r^{(0)} \cup Q_r^{(1)}$, with the associated index set being $S=[r]$. 
We first prove $\sigma \in \nicea_r$.
A point $x\in \sigma$ is of the form $x=y \times \{i\}$ with $y\in \tau, i\in \{0,1\}$.
As $\bar x = \bar y \times \{1-i\}$ and $y\in \nicea_{r-1}$, we see that $x\in \sigma$ iff $\bar x \notin \sigma$.
 
We proceed by determining the local diameter.
Take $x=y \times \{i\} \in \sigma$.
As $\locdiam(\tau,y)=r-1$, there exists $y'\in \tau$ with $d(y,y')=r-1$.
But then $y' \times \{1-i\} \in \sigma$ and $d\left(y \times \{i\},y' \times \{1-i\}\right)=r$.

It remains to show that $\chull(\sigma)=Q_{r+1}$.
Similarly as in the proof of the even case, this follows from the fact that as $r \geq 3$ we can prescribe any single coordinate as we please,
and then fill in the rest of the coordinates to obtain a vertex of $\sigma$:
\begin{itemize}
\item the last coordinate can be chosen freely by the construction of $\sigma$; 
\item any of the first $r$ coordinates can be choosen freely by the construction and by the even case.
\end{itemize}
\end{proof}

We are now in position to prove the main result of this section, Theorem~\ref{ThmMain1}, which states that $\rank H_{2^r-1}( \vr{Q_{n}}{r}) \geq 2^{n-(r+1)}\binom{n}{r+1}$.

\begin{proof}[Proof of Theorem~\ref{ThmMain1}]
For notational convenience, let $k = 2^{n-(r+1)}\binom{n}{r+1}$.
There are $k$ isometric copies of $Q_{r+1}$ in $Q_n$ obtained via embeddings $\iota$, which we enumerate as $C_1, C_2, \ldots, C_{k}$.
For each $i$:
\begin{enumerate}
\item Let $\sigma_i$ be the maximal simplex in $C_i$ obtained from Proposition~\ref{PropReduction} and Lemma~\ref{LemmaREvenAndOdd}.
\item Let $[\alpha_i]$ be the mentioned cross-polytopal generator of $H_{2^r -1}(\vr{C_i}{r})$, and recall the coefficient of $\sigma_i$ in $\alpha_i$ is $1$.
\item Let $\omega_i$ be the $(2^r - 1)$-cochain on $Q_n$ mapping $\sigma_i$ to $1$ and the rest of the $(2^r-1)$-dimensional simplices to $0$.
As $\sigma_i$ is a maximal simplex in $Q_n$, the cochains $\omega_i$ are cocycles.
 \item Note that $\sigma_i$ is not contained as a term in $\alpha_j$ for any $i \neq j$.
Indeed, if that was the case, $\sigma_i$ would be contained in the lower-dimensional cube $C_i \cap C_j$, contradicting the conclusion $\chull(\sigma)=Q_{r+1}$ from Lemma~\ref{LemmaREvenAndOdd}.
As a result, $[\omega_i] \frown [\alpha_i]=1$ and $[\omega_i] \frown [\alpha_j]=0$ for $i \neq j$.
\end{enumerate}
It remains to prove that homology classes $[\alpha_i]$ in $H_{2^r-1}( \vr{Q_{n}}{r})$ (via the natural inclusion) are linearly independent.
If 
$\sum_{i=1}^{k} \lambda_i [\alpha_i]=0$
for some $\lambda_i \in \mathbb{Z}$, then applying $[\omega_j]$ via the cap product we obtain $\lambda_j=0$ by (4) above.
Hence the rank of $H_{2^r-1}( \vr{Q_{n}}{r})$ is at least $k=2^{n-(r+1)}\binom{n}{r+1}$.
\end{proof}

\section{More contractions: concentrations}
\label{sec:concentrations}

Up to now the only contractions that we have utilized are the projections $\pi_S$.
In order to establish additional homology bounds we need to employ a new kind of contractions called concentrations.
The idea of a such maps in low-dimensional settings is shown in Figures~\ref{Fig1} and~\ref{Fig3}.
We proceed with an explanation of the general case.

\begin{figure}[htb]
\centering
\includegraphics[width=4in]{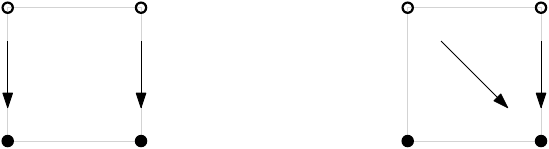}
\caption{Two contractions on $Q_2$: a projection $\pi_{[1]}$ and a concentration map.}
\label{Fig1}
\end{figure}

Let $n > k$ be positive integers.
Choose $a=(a_{k+1}, a_{k+2}, \ldots, a_n) \in \{0,1\}^{n-k}$ and let $C$ denote the copy of $Q_k$ identified as $Q_k^a$, i.e., 
\[
C= \{0,1\}^k \times \{a_{k+1}\} \times \{a_{k+2}\} \times \ldots \times \{a_{n}\}.
\]

Working towards a contraction we define a \emph{concentration map $f\colon Q_n \to C$ of codimension $n-k$} by the following rule:
\begin{enumerate}
    \item $f|_C = \textrm{Id}_C$, and
    \item for $x=(x_1, x_2, \ldots, x_n)\in Q_n \setminus C$ we define 
    \[
    f(x)=(x_1, x_2, \ldots, x_{k-1}, 1, a_{k+1}, a_{k+2}, \ldots, a_n).
    \]
\end{enumerate}
In particular, we concentrate the $k$-th coordinate of $Q_n \setminus C$ to $1$ (although we might as well have used $0$).
Permuting the coordinates of $Q_n$ generates other concentration maps.
In order to discuss the properties of concentration functions it suffices to consider the concentrations defined as $f$ above.

\begin{figure}[htb]
\centering
\includegraphics[width=4in]{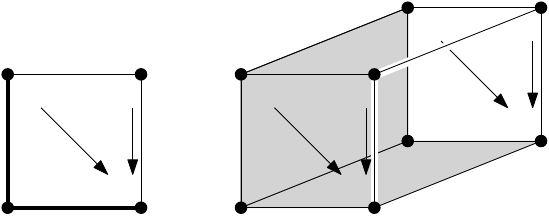}
\caption{Two concentrations of codimension one.
In both cases there are two codimension one subcubes that are being mapped isometrically (the thick $Q_1$ on the left and the shaded $Q_2$ on the right).
On the left we have $n=2$, $k=1$, and $a_2=0$, and on the right we have $n=3$, $k=2$, and $a_3=0$.}
\label{Fig3}
\end{figure}

\begin{proposition}
 \label{PropConcentration}
 Let $f$ be the concentration map as defined above.
 \begin{enumerate}
     \item[(i)] Map $f$ is a contraction.
     \item[(ii)] The $n-k+1$ many $k$-dimensional cubes $Q_{k}\subseteq Q_n$ containing the $(k-1)$-dimensional cube
     \[
     \{0,1\}^{k-1} \times \{0\} \times \{a_{k+1}\} \times \{a_{k+2}\} \times \ldots \times \{a_{n}\}
     \]
     are mapped onto $C$ isometrically by $f$.
     \item[(iii)] Map $f$ maps all other $k$-dimensional cubes $Q_{k}\subseteq Q_n$ onto cubes of dimension less than $k$.
 \end{enumerate}
\end{proposition}

\begin{proof}
(i) We verify the claim by a case analysis:

For $x,y\in C$ we have 
    \[
    d(f(x),f(y))=d(x,y)
    \]
    as $f$ is the identity on $C$.

For $x,y\notin C$ the quantity $d(x,y)$ is the number of coordinates in which $x$ and $y$ differ, while $d(f(x),f(y))$ is the number of coordinates amongst the first $k-1$ in which $x$ and $y$ differ.
    Thus $d(f(x),f(y)) \leq d(x,y)$.

Let $x\in C, y\notin C$.
    Then $d(x,y)$ is the sum of the following two numbers:
    \begin{itemize}
        \item The number of coordinates amongst the first $k$ coordinates in which $x$ and $y$ differ.
        \item The number of coordinates amongst the last  $n-k$ coordinates in which $x$ and $y$ differ.
        Note that this quantity is at least $1$ as $y \notin C$.
    \end{itemize}
    On the other side, $d(f(x),f(y))$ is less than or equal to the sum of the following two numbers:
    \begin{itemize}
        \item The number of coordinates amongst the first $k-1$ coordinates in which $x$ and $y$ differ.
        \item The number $1$ if the $k$-th coordinate of $x$ does not equal $1$.
    \end{itemize}
    Together we obtain $d(f(x),f(y)) \leq d(x,y)$.

This covers all possible cases.
We conclude that $f$ is a contraction.
    
(ii) The $n-k+1$ copies of $Q_k$ in question are the ones of the form
\[
\{0,1\}^{k-1} \times \{0\} \times \{a_{k+1}\} \times \{a_{k+2}\} \times \ldots 
\times \{a_{p-1}\} \times \{0,1\} \times \{a_{p+1}\} \times \ldots
\times \{a_{n}\}
\]
for $p\in \{k, k+1, \ldots, n\}$.
The case $p=k$ shows that $C$ is one of these copies.
Note that the part  
\[
\{0,1\}^{k-1} \times \{0\} \times \{a_{k+1}\} \times \{a_{k+2}\} \times \ldots 
\times \{a_{p-1}\} \times \{a_p\} \times \{a_{p+1}\} \times \ldots
\times \{a_{n}\}
\]
is contained in $C$ and thus $f$ is identity on it.
On the other hand, the part 
\[
\{0,1\}^{k-1} \times \{0\} \times \{a_{k+1}\} \times \{a_{k+2}\} \times \ldots \times \{a_{p-1}\} \times \{1- a_p\} \times \{a_{p+1}\} \times \ldots
\times \{a_{n}\}
\]
gets mapped to 
\[
\{0,1\}^{k-1} \times \{1\} \times \{a_{k+1}\} \times \{a_{k+2}\} \times \ldots 
\times \{a_{p-1}\} \times \{a_p\} \times \{a_{p+1}\} \times \ldots
\times \{a_{n}\}
\]
by retaining the first $k-1$ coordinates.
Together these two parts form $C$.

(iii) Any $k$-dimensional cube $Q_{k}\subseteq Q_n$ not mentioned in (2) has one of the first $k-1$ coordinates (say, the $p$-th coordinate) constant.
Thus the same holds for its image via $f$ and consequently, its image is contained in the corresponding $Q_{k-1}\subseteq Q_n$, i.e., the one having the $p$-th coordinate constant, and having the last $n-k$ coordinates prescribed as in $C$.
\end{proof}

\section{Homology bounds via contractions}
\label{sec:contractions-homology}

\begin{figure}
\centering
\includegraphics[width=4in]{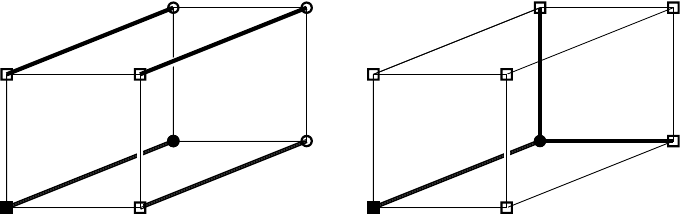}
\caption{\emph{(Left)} The projection onto the front bottom $Q_1$ isometrically identifies the four bold copies of $Q_1$, and sends other copies of $Q_1$ to a point.
\emph{(Right)} The concentration onto the bottom left $Q_1$, mapping all hollow-square vertices to the solid square vertex, isometrically identifies the three bold copies of $Q_1$, and sends the other copies of $Q_1$ to a point.
In this case the concentration is of codimension $2$, i.e., maps $Q_3 \to Q_1$.
As the codimension $t$ increases, the number of isometrically identified subcubes is exponential $2^t$ for projections and linear $t+1$ for concentrations.
The exponential increase leads to weaker lower bound in Theorem \ref{ThmMain2} than the linear increase in Theorems \ref{ThmMain3} and \ref{ThmMain4}.
}
\label{Fig4}
\end{figure}

In Section~\ref{sec:maximal} we showed how the appearance of cross-polytopal homology classes in certain dimensions of the Vietoris--Rips complexes of cubes generate independent homology elements in Vietoris--Rips complexes of higher-dimensional cubes.
Applying Proposition~\ref{PropContrEmbed} to the canonical projections implies that the homology of each smaller subcube embeds.
However, the independence of homology classes arising from various subcubes was proved using maximal simplices; this argument depended heavily on the fact that convenient (cross-polytopal) homology representatives were available to us.
In this section we aim to provide an analogous result for homology in any dimension, without prior knowledge of homology generators.

\begin{example}
\label{Ex1}
As a motivating example, consider the graph $\vr{Q_3}{1}$.
The cube $Q_3$ contains six subcubes $Q_2$ and each $\vr{Q_2}{1}\simeq S^1$ has the first Betti number equal to $1$.
However, the first Betti number of $Q_3$ is not $6$ but rather $5$, demonstrating that homology classes of various subcubes might in general interfere, i.e., not be independent.
In Theorem~\ref{ThmMain1} we proved there is no such interference in a specific case (involving cross-polytopes).
In the more general case of this section, we prove lower bounds even when independence may not hold.
\end{example}

\textbf{The general setting of this section}: Fix $r, q\in \{1,2,\ldots\}$.
Let $p\ge 1$ be the smallest integer for which $H_q(\vr{Q_p}{r})\neq 0$.
We implicitly assume that such a $p$ exists, i.e., that $H_q(\vr{Q_n}{r})$ is non-trivial for some $n$.

The cube $Q_n$ contains $2^{n-p} \binom{n}{p}$ canonical $Q_p$ subcubes.
We aim to estimate the rank of the homomorphism
\[
\bigoplus_{{2^{n-p} \binom{n}{p}}}
H_q(\vr{Q_p}{r}) \to H_q(\vr{Q_n}{r})
\]
induced by the map
\[
\coprod_{2^{n-p} \binom{n}{p}} Q_p \to Q_n,
\]
consisting of the natural inclusions of all of the $Q_p$ subcubes.

\subsection{Homology bounds via projections}

We will start with the simplest argument to demonstrate how contractions, in this case projections, may be used to lower bound the homology.

\begin{theorem}
\label{ThmMain2}
Let $q\ge 1$.
If $p$ is the smallest integer for which $H_q(\vr{Q_p}{r})\neq 0$, then for $n \geq p$,
\[
\rank H_{q}( \vr{Q_{n}}{r}) \geq \binom{n}{p} \cdot \rank H_{q}( \vr{Q_{p}}{r}).
\]
\end{theorem}

\begin{proof}
Let the $Q_p^*$ denote the $\binom{n}{p}2^{n-p}$ different $p$-dimensional subcubes of $Q_n$ (say as $*$ varies from $1$ to $\binom{n}{p}2^{n-p}$).
For a subset $S \subseteq [n]$ of cardinality $p$, the projection $\pi_S \colon Q_n \to Q_p$ is an isometry on $2^{n-p}$ of these subcubes $Q_p^*$ (the ones having exactly the coordinates $S$ as the free coordinates).
For each such $S$ choose one of these cubes and designate it as $Q_{p,S}$, thus marking $\binom{n}{p}$ copies of $Q_p \subseteq Q_n$.
For each $S$ let $a_{1,S}, a_{2,S}, \ldots, a_{\rank H_{q}( \vr{Q_{p}}{r}),S}$ denote a largest linearly independent collection in $H_{q}( \vr{Q_{p,S}}{r})$.
We claim that the collection $\{a_{i,S}\}$ of cardinality $\binom{n}{p}\cdot \rank H_{q}( \vr{Q_{p}}{r})$ is linearly independent in $H_{q}( \vr{Q_{n}}{r})$.

Assume 
\[
\sum_{i,S}\lambda_{i,S} \cdot  a_{i,S}=0
\]
for some coefficients $\lambda_{i,S}$.
Fix a subset $S' \subseteq [n]$ of cardinality $p$, and to the equality above apply the map on $H_q$ induced by the projection $\pi_{S'}\colon Q_n \to Q_p$.
As $p$ is the minimal dimension of a cube in which $H_q$ is non-trivial on Vietoris--Rips complexes, and as $\pi_{S'}$ maps all of the $Q_p^*$ to a smaller-dimensional cube except for the ones with exactly the coordinates in $S'$ as the free coordinates, 
we obtain $(\pi_{S'})_*(a_{i,S})=0$ for all $S \neq S'$,
where $*$ denotes the induced map on homology.
As $\pi_{S'}|_{Q_{p,S'}}\to Q_p$ is a bijection on the corresponding Vietoris--Rips complexes, and induces an isomorphism homology, we have reduced our equation to
\[
\sum_{i}\lambda_{i,S'} \cdot  a_{i,S'}=0.
\]
Consequently, $\lambda_{i,S'}=0$ for all $i$, as $\{a_{i,S'}\}$ forms a linearly independent collection in $H_{q}( \vr{Q_{p}}{r})$ by definition.
As $S'$ was arbitrary we conclude $\lambda_{i,S}=0$ for all $i$ and $S$, and thus the claim holds.
\end{proof}

\subsection{Codimension $1$ homology bounds via concentrations}

The following result states that in codimension $1$ (i.e., when we increase the dimension of the cube by $1$ from $p$ to $p+1$), all but at most one of the subcubes induce independent inclusions on homology.
Example~\ref{Ex1} shows that all subcubes need not induce independent inclusions of homology (and we see that Theorem~\ref{ThmMain3} is tight in the case of Example~\ref{Ex1}).

\begin{theorem}
\label{ThmMain3}
Let $q\ge 1$.
If $p$ is the smallest integer for which $H_q(\vr{Q_p}{r})\neq 0$, then
\[
\rank H_{q}( \vr{Q_{p+1}}{r}) \geq 
(2p+1)
\cdot \rank H_{q}( \vr{Q_{p}}{r}).
\]
\end{theorem}

\begin{proof}
The cube $Q_{p+1}$ contains 
$2p+2$ subcubes $Q_p$, which we enumerate as $Q_{p,1}, Q_{p,2}, \ldots, 
Q_{p,2p+2}$.
For each $1 \le j \le 
2p+1$, let $\{a_{i,j} \mid 1\le i \le \rank H_{q}( \vr{Q_{p}}{r})\}$ denote a largest linearly independent collection in $H_{q}( \vr{Q_{p,j}}{r})$.
We claim that the collection 
\[
\left\{a_{i,j}\mid 1\le i \le \rank H_{q}( \vr{Q_{p}}{r}), 1\le j\le 2p+1\right\}
\]
of cardinality 
$(2p+1)\cdot \rank H_{q}( \vr{Q_{p}}{r})$ is linearly independent.

Assume 
\begin{equation}
\label{Eq2}
\sum_{i,j}\lambda_{i,j} \cdot  a_{i,j}=0    
\end{equation}
for some coefficients $\lambda_{i,j}$ with $1\le i \le \rank H_{q}( \vr{Q_{p}}{r})$ and $1\le j\le 2p+1$.
Note that there are no representatives from $Q_{p,2p+2}$.
Choose a subcube $Q_{p-1}\subseteq Q_{p,2p+2}$.
It is the intersection of $Q_{p,2p+2}$ and another $p$-dimensional cube of the form $Q_{p,*}$, say $Q_{p,1}$.
We apply the concentration map $f$ corresponding to these choices using Proposition~\ref{PropConcentration}:
\begin{enumerate}
\item[(i)] $f$ is bijective on exactly two $p$-cubes: $Q_{p,2p+2}$ and $Q_{p,1}$.
\item[(ii)] $f$ maps all other $p$-cubes to cubes of dimension less than $p$.
By the choice of $p$ (as the first dimension of a cube in which nontrivial $p$-dimensional homology appears in its Vietoris--Rips complex), we get $f_*(a_{i,j})=0$ for all $i>1$.
\item[(iii)] As a result, after applying the induced map $f_*$ on homology, Equation~\eqref{Eq2} simplifies to 
\[
\sum_{j}\lambda_{j} \cdot  a_{1,j}=0.
\]
Consequently, $\lambda_{1,j}=0$ for all $i$, as $\{a_{1,j}\}$ forms a linearly independent collection in $H_{q}( \vr{Q_{p,1}}{r})$ by definition.
\end{enumerate}

We keep repeating the procedure of the previous paragraph:
\begin{itemize}
    \item Choose a subcube $Q_{p,j}$, whose corresponding coefficients $\lambda_{i,j}$ have been determined to be zero, and choose a neighboring subcube $Q_{p,j'}$ (i.e., a subcube with a common $(p-1)$-dimensional cube), whose corresponding coefficients $\lambda_{i,j'}$ have not yet been determined to be zero.
    \item Apply the concentration map corresponding to these two $p$-dimensional cubes to deduce that the coefficients $\lambda_{i,j'}$ also equal zero.
\end{itemize}
Any cube $Q_{p,j'}$ can be reached from  $Q_{p,2p+2}$ by an appropriate sequence of cubes (i.e., $Q_{p,2p+2}$, a $(p-1)$-dimensional subcube thereof, an enclosing $p$-dimensional cube, a $(p-1)$-dimensional subcube thereof, $\ldots, Q_{p,j'}$).
Therefore, we can eventually deduce that $\lambda_{i,j}=0$ for all $i$ and $j$.
Hence the rank bound holds due to the setup of Equation~\eqref{Eq2}.
\end{proof}

\subsection{General homology bounds via concentrations}
\label{ssec:gen-via-concentrations}

In this subsection we will generalize the argument of Theorem~\ref{ThmMain3} to deduce a lower bound on homology on all subsequent larger (not just codimension one) cubes.

The core idea is the following.
In the previous subsection we were able to ``connect'' $p$-dimensional cubes by concentrations.
Each chosen concentration was bijective on exactly two adjacent cubes of the form $Q_p$ sharing a common $(p-1)$-dimensional cube; see item (i) in the proof of Theorem~\ref{ThmMain3}.
If the coefficients of Equation~\eqref{Eq2} corresponding to one of the two copies of $Q_p$ were known to be trivial, then the homological version of the said concentration transformed Equation~\eqref{Eq2} so that only the coefficients corresponding to the other of the two cubes $Q_p$ were retained; see item (ii).
These coefficients were then deduced to be trivial by their definition as the resulting equation contained only terms arising from a single $Q_p$, see item (iii).

In this subsection we generalize this argument to all higher-dimensional cubes.
Instead of concentrations isolating $2$ adjacent copies of $Q_p$ (as happens in codimension $1$), the concentrations will in general isolate $n-p+1$ (i.e., codimension plus one; see Proposition~\ref{PropConcentration}(ii)) copies of $Q_p$.
The main technical question is thus to determine how many of the subcubes are independent in the above sense.

\begin{theorem}
\label{ThmMain4}
Let $q\ge 1$.
If $p$ is the smallest integer for which $H_q(\vr{Q_p}{r})\neq 0$, then for $n \geq p$,
\[
\rank H_{q}( \vr{Q_{n}}{r}) \geq 
\sum_{i=p}^n 2^{i-p} \binom{i-1}{p-1}
\cdot \rank H_{q}( \vr{Q_{p}}{r}).
\]
\end{theorem}

In particular cases, Theorem~\ref{ThmMain4} reduces to the following.
For $n=p$ we get the tautology that $\rank H_{q}( \vr{Q_{p}}{r})$ is at least as large as itself.
For $n=p+1$ we recover Theorem~\ref{ThmMain3}:
\begin{align*}
\rank H_{q}( \vr{Q_{n}}{r})
&\geq 
\rank H_{q}( \vr{Q_{p}}{r})+ 2 \binom{p}{p-1}
\cdot \rank H_{q}( \vr{Q_{p}}{r}) \\
&= (2p+1)\cdot \rank H_{q}( \vr{Q_{p}}{r}).
\end{align*}

\begin{figure}[htb]
\centering
\includegraphics[width=4in]{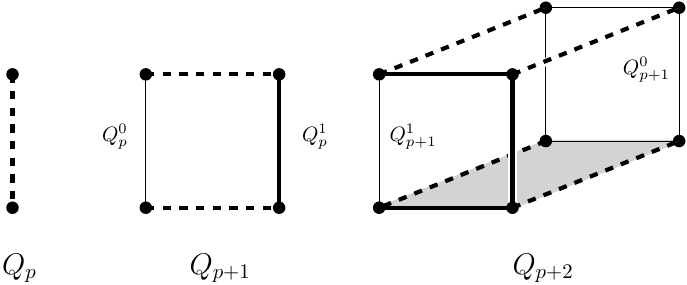}
\caption{A sketch of the proof of Theorem~\ref{ThmMain4}.
In each step the thick dashed lines represent copies of $Q_p$ in $M_n$ yielding new independent homology classes within the Vietoris-Rips complex of $Q_n$, in addition to the established independent homology classes (thick solid lines) arising from certain copies of $Q_p$, denoted by $F_n$, within the front face $Q_{n-1}^1$.
The multiplicative factor in the theorem is the total number of the thick edges, both dashed and solid ones.}
\label{Fig5}
\end{figure}

\begin{proof}[Proof of Theorem~\ref{ThmMain4}]
The cube $Q_n$ consists of two disjoint copies of $Q_{n-1}$; see Figure~\ref{Fig5}:
\begin{itemize}
    \item the rear one with the last coordinate $0$, denoted by $Q_{n-1}^0$, and
    \item the front one with the last coordinate $1$, denoted by $Q_{n-1}^1$.
\end{itemize}
We partition the $Q_p$ subcubes of $Q_n$ into three classes:
\begin{itemize}
    \item The ones contained in the rear $Q_{n-1}^0$ where vertices have last coordinate $0$, denoted by $R_n$.
    \item The ones contained in the front $Q_{n-1}^1$ where vertices have last coordinate $1$, denoted by $F_n$.
    \item The ones contained in the middle passage between them, denoted by $M_n$.
    Each such $Q_p$ in $M_n$ is of the form $D \times \{0,1\}$, where $D \subseteq Q_{n-1}$ is a copy of $Q_{p-1}$.
\end{itemize}

We will prove that the following $Q_p$ subcubes of $Q_n$ induce independent embeddings on homology $H_q$ of Vietoris--Rips complexes: 
the elements of $M_n$ (dashed cubes in Figure~\ref{Fig5}) and the elements of $F_n$ that have inductively been shown to include independent embeddings on homology $H_q$ of Vietoris--Rips complexes on $Q^1_{n-1}$ (bold cubes in Figure~\ref{Fig5}).
The initial cases of the inductive process have been discussed in the paragraph before the proof.
For $n=p+1$ this is Theorem~\ref{ThmMain3}.

The cardinality of $M_n$ is $2^{(n-1)-(p-1)} \binom{n-1}{p-1} = 2^{n-p} \binom{n-1}{p-1}$, which is the number of $Q_{p-1}$ subcubes in $Q_{n-1}^0$.
Each such subcube has the last coordinate constantly $0$.
Taking a union with a copy of the same $Q_{p-1}$ subcube with the last coordinates changed to $1$, we obtain a $Q_p$ subcube in $M_n$.
It is apparent that all elements of $M_n$ arise this way.
Let us enumerate the elements of $M_n$ as $Q_{p,j}^M$ with $1\le j \le 2^{n-p} \binom{n-1}{p-1}$.
For each such $j$ let $\{a_{i,j}\mid 1\le i \le \rank H_{q}( \vr{Q_{p}}{r})\}$ denote a largest linearly independent collection in $H_{q}( \vr{Q_{p,j}^M}{r})$.

The cardinality of the copies of $Q_p$ in $F_n$ that have inductively been shown to include independent embeddings on homology $H_q$ of Vietoris--Rips complexes on $Q^1_{n-1}$ equals $\sum_{i=p}^{n-1} 2^{i-p} \binom{i-1}{p-1}$, by inductive assumption.
Let us enumerate them by $Q_{p,j}^F$ with $1\le j\le \sum_{i=p}^{n-1} 2^{i-p} \binom{i-1}{p-1}$.
For each such $j$ let $\{b_{i,j} \mid 1\le i\le \rank H_{q}( \vr{Q_{p}}{r})\}$ denote a largest linearly independent collection in $H_{q}( \vr{Q_{p,j}^F}{r})$.

Assuming the equality 
\begin{equation}
\label{Eq3a}
\sum_{i,j}\lambda_{i,j} \cdot  a_{i,j} + \sum_{i,j}\mu_{i,j} \cdot  b_{i,j}=0    
\end{equation}
in $H_{q}( \vr{Q_{n}}{r})$ for some coefficients $\lambda_{i,j}$, $\mu_{i,j}$, we claim that all coefficients equal zero.
This will prove the theorem as the number of involved terms equals $\sum_{i=p}^n 2^{i-p} \binom{i-1}{p-1}
\cdot \rank H_{q}( \vr{Q_{p}}{r})$.

We will first prove that the coefficients $\lambda_{i,j}$ are all zero.
Fix some $1\le j \le 2^{n-p} \binom{n-1}{p-1}$, and let $D$ be the copy of $Q_{p-1}$ in $Q_{n-1}^0$ so that $C\coloneqq D \times \{0,1\}$ is equal to $Q_{p,j}^M$.
Let $f$ be any concentration $Q_n \to C:=D \times \{0,1\}$.
(For example, in Figure~\ref{Fig4} one can visualize $D$ as the solid round vertex, and $C$ as the edge between the two solid vertices.)
By Proposition~\ref{PropConcentration}:
\begin{itemize}
    \item $f$ maps any $Q_p$ subcube of $Q_n$ that contains $D$ bijectively onto $C=Q_{p,j}^M$.
    All such subcubes except for $C$ are contained in $R_n$.
    \item $f$ maps all of the other $Q_p$ subcubes of $Q_n$ to lower-dimensional subcubes.
\end{itemize}
These two observations imply that applying the induced map $f_*$ on homology to Equation~\eqref{Eq3a}, we obtain 
$
\sum_{i}\lambda_{i,j} \cdot  a_{i,j}=0.
$
By the choice of $\{a_{i,j}\}_i$ as an independent collection of homology classes for $H_{q}(\vr{Q_{p,j}^M}{r})$, we obtain $\lambda_{i,j}=0$ for all $i$.
Since this can be done for any $1\le j \le 2^{n-p} \binom{n-1}{p-1}$, we have $\lambda_{i,j}=0$ for all $i$ and $j$.

We have thus reduced Equation~\eqref{Eq3a} to 
$
\sum_{i,j}\mu_{i,j} \cdot  b_{i,j}=0.
$
Let $\pi_S \colon Q_n \to Q_{n-1}$ be the projection that forgets the last coordinate of each vector (explicitly, $S=[n-1]\subseteq [n]$).
Note that the restrictions of $\pi_S$ to $Q_n^0$ and to $Q_n^1$ are bijections.
Hence, after applying the induced map $(\pi_S)_*$ on homology, the inductive definition of the $b_{i,j}$ implies that $\mu_{i,j}=0$ for all $i$ and $j$.

\end{proof}

\subsection{Comparison with known results}
\label{SubsComparison}

In this subsection we demonstrate that our lower bounds agree with actual ranks of homology in many known cases.
In particular, for $r=1,2,3$, if we assume that we know the homotopy types or Betti numbers for the first few cases ($n\le r+1$ or $n \le r+2$), then we show that our lower bounds on the Betti numbers of $\vr{Q_n}{r}$ are tight (i.e.\ optimal) for all $n$ and for all dimensions of homology.
For $r=4$ we explain the best lower bounds we know on the Betti numbers of $\vr{Q_n}{4}$, which are based on homology computations by Ziqin Feng.
Since the homotopy types of $\vr{Q_n}{4}$ are unknown for $n\ge 6$, we do not know if these bounds are tight. 
For a summary see Table \ref{table:homotopy-types2}.

\subsubsection{The case $\mathbf{r=1}$}

Assuming the obvious homeomorphism $\vr{Q_2}{1} \cong S^1$, the lower bound of Theorem~\ref{ThmMain4} with $p=2$ gives 
\[
\rank H_1(\vr{Q_n}{1})
\ge \sum_{i=2}^n 2^{i-2} \binom{i-1}{1}
=\sum_{i=2}^n (i-1) 2^{i-2}
=n 2^{n-1} - 2^n + 1, 
\]
where the last step is explained in Appendix~\ref{app:1}.
This inequality is actually an equality, as one can see via an Euler characteristic computation.
Indeed, $\vr{Q_n}{1}$ has $2^n$ vertices and $n2^{n-1}$ edges, and so the Euler characteristic is $2^n - n2^{n-1}$.
As $\vr{Q_n}{1}$ is connected, the rank of $H_1(\vr{Q_n}{1})$ equals $n 2^{n-1} - 2^n + 1$.
See~\cite[Proposition~4.12]{carlsson2020persistent} for a related computation.



\subsubsection{The case $\mathbf{r=2}$}

We know that the embedding of each individual subcube induces an injection on homology.
Our results provide the lower bounds on the rank of the map on homology induced by the inclusion of \emph{all} subcubes $Q_p$ (where $p$ is the dimension of the first appearance of $q$-dimensional homology $H_q$).
The upper bound for homology obtained in this way is $2^{n-p}\binom{n}{p} \rank H_q(\vr{Q_p}{r})$, where the the multiplicative constant is the number of all $Q_p$ subcubes on $Q_n$.
These possible generators are all independent in the case of cross-polytopal generators (Theorem~\ref{ThmMain1}).
In case this bound is exceeded we can thus deduce that certain new homology classes appear that are \emph{not} generated by the embeddings of $Q_p$ subcubes.

\begin{table}[bht]
{\footnotesize
\def\arraystretch{1.2}
\begin{tabular}{| >{$} c <{$} || >{$} c <{$} | >{$} c <{$} || >{$} c <{$} | >{$} c <{$} | }
\hline
_n \backslash ^r  & 1 & 1 & 2 & 2 \\
\hline
1  & * & * & * & * \\
\hline
2 & \textcolor{violet}{\mathbf{S^1}} & \beta_{1}\geq \textcolor{violet}{1} & * & * \\
\hline
3 & \bigvee^5 S^1 & \beta_{1}\geq \textcolor{violet}{ 5} & \textcolor{red}{\mathbf{S^3}} &  \beta_{3}\geq \textcolor{red}{1} \\
\hline
4 & \bigvee^{17} S^1 & \beta_{1}\geq  \textcolor{violet}{17} & \bigvee^8 S^3 \vee \textcolor{blue}{\mathbf{S^3}} & \beta_{3}\geq \textcolor{red}{ 8} + \textcolor{blue}{1} \\
\hline
5 & \bigvee^{49} S^1 & \beta_{1}\geq  \textcolor{violet}{49} & \bigvee^{49} S^3 & \beta_{3}\geq \textcolor{red}{ 40} + \textcolor{blue}{9} \\
\hline
6 & \bigvee^{129} S^1 & \beta_{1}\geq  \textcolor{violet}{129} & \bigvee^{209} S^3 & \beta_{3}\geq \textcolor{red}{ 160} + \textcolor{blue}{49} \\
\hline
7 & \bigvee^{321} S^1 & \beta_{1}\geq  \textcolor{violet}{321} & \bigvee^{769} S^3 & \beta_{3}\geq \textcolor{red}{ 560} + \textcolor{blue}{209} \\
\hline
8 & \bigvee^{769} S^1 & \beta_{1}\geq  \textcolor{violet}{769} & \bigvee^{2561} S^3 & \beta_{3}\geq \textcolor{red}{ 1792} + \textcolor{blue}{769} \\
\hline
9 & \bigvee^{1793} S^1 & \beta_{1}\geq  \textcolor{violet}{1793} & \bigvee^{7937} S^3 & \beta_{3}\geq \textcolor{red}{ 5376} + \textcolor{blue}{2561} \\
\hline
10 & \bigvee^{4097} S^1 & \beta_{1}\ge  \textcolor{violet}{4097} & \bigvee^{23297} S^3 & \beta_{3}\geq \textcolor{red}{ 15360} + \textcolor{blue}{7937} \\
\hline
11 & \bigvee^{9217} S^1 & \beta_{1}\ge  \textcolor{violet}{9217} & \bigvee^{65537} S^3 & \beta_{3}\geq \textcolor{red}{ 42240} + \textcolor{blue}{23297} \\
\hline
12 & \bigvee^{20481} S^1 & \beta_{1}\ge \textcolor{violet}{20482} & \bigvee^{178177} S^3 & \beta_{3}\geq \textcolor{red}{ 112640} + \textcolor{blue}{65537} \\
\hline\vdots&\vdots&\vdots&\vdots&\vdots\\\hline\hline
_n \backslash ^r  & 3 & 3 & 4 & 4 \\
\hline
1  & * & * & * & * \\
\hline
2 & * & * & * & * \\
\hline
3 & * & * & * & * \\
\hline
4 & \textcolor{red}{\mathbf{S^7}} & \beta_{7}\geq \textcolor{red}{1} & * & * \\
\hline
5 & \bigvee^{10} S^7 \vee \textcolor{blue}{\mathbf{S^4}} & \beta_{7}\geq  \textcolor{red}{10}, \beta_{4}\geq  \textcolor{blue}{1} & \textcolor{red}{\mathbf{S^{15}}} & \beta_{15}\geq \textcolor{red}{1}\\
\hline
6 & \bigvee^{60} S^7 \vee \bigvee^{11} S^4 & \beta_{7}\geq  \textcolor{red}{60}, \beta_{4}\geq  \textcolor{blue}{11} & \beta_{15}=\textcolor{red}{12}+\textcolor{gray}{\mathbf{2}}, \beta_7=\textcolor{brown}{\mathbf{239}} & \beta_{15}\ge\textcolor{red}{12}+\textcolor{gray}{2}, \beta_7\ge\textcolor{brown}{239} \\
\hline
7 & \bigvee^{280} S^7 \vee \bigvee^{71} S^4 & \beta_{7}\geq  \textcolor{red}{280}, \beta_{4}\geq  \textcolor{blue}{71} & ? & \beta_{15}\ge\textcolor{red}{84}+\textcolor{gray}{26}, \beta_7\ge\textcolor{brown}{3107} \\
\hline
8 & \bigvee^{1120} S^7 \vee \bigvee^{351} S^4 & \beta_{7}\geq  \textcolor{red}{1120}, \beta_{4}\geq  \textcolor{blue}{351} & ? & \beta_{15}\ge\textcolor{red}{448}+\textcolor{gray}{194}, \beta_7\ge\textcolor{brown}{23183} \\
\hline
9 & \bigvee^{4032} S^7 \vee \bigvee^{1471} S^4 & \beta_{7}\geq  \textcolor{red}{4032}, \beta_{4}\geq  \textcolor{blue}{1471} & ? & \beta_{15}\ge\textcolor{red}{2016}+\textcolor{gray}{1090}, \beta_7\ge\textcolor{brown}{130255} \\
\hline
10 & \bigvee^{13440} S^7 \vee \bigvee^{5503} S^4 & \beta_{7}\geq  \textcolor{red}{13440}, \beta_{4}\geq  \textcolor{blue}{5503} & ? & \beta_{15}\ge\textcolor{red}{8064}+\textcolor{gray}{5122}, \beta_7\ge\textcolor{brown}{612079} \\
\hline
11 & \bigvee^{42240} S^7 \vee \bigvee^{18943} S^4 & \beta_{7}\geq  \textcolor{red}{42240}, \beta_{4}\geq  \textcolor{blue}{18943} & ? & \beta_{15}\ge\textcolor{red}{29568}+\textcolor{gray}{21250}, \beta_7\ge\textcolor{brown}{2539375} \\
\hline
12 & \bigvee^{126720} S^7 \vee \bigvee^{61183} S^4 & \beta_{7}\geq  \textcolor{red}{126720}, \beta_{4}\geq  \textcolor{blue}{61183} & ? & \beta_{15}\ge\textcolor{red}{101376}+\textcolor{gray}{80386}, \beta_7\ge\textcolor{brown}{9606127} \\
\hline
\end{tabular}
}
\caption{For each of $r=1,2,3,4$, the pair of columns represent the comparison between \emph{(left)} known homotopy types and homology groups of $\vr{Q_n}{r}$ with \emph{(right)} the bounds arising from our results.
The bold red spheres are the initial cross-polytopal spheres that induce the red lower bounds on Betti numbers due to Theorem \ref{ThmMain1}.
For $r=1,2,3$, the bold blue and violet spheres induce the blue and violet lower bounds on Betti numbers due to Theorems~\ref{ThmMain4} and~\ref{ThmMain5}.
Observe that the total lower bounds match the known Betti numbers for $r=1,2,3$.
The homology computations for $\vr{Q_6}{4}$ by Ziqin Feng induce the lower bounds on Betti numbers of $\vr{Q_n}{4}$ for $n\ge 6$ by Theorems~\ref{ThmMain4} and~\ref{ThmMain5}.
Theorem~\ref{ThmGeom}(\textbf{iii}) states that in each column $r\ge 2$, we have at least one homology class (such as the features in blue) that is not induced from a red cross-polytopal sphere.}
\label{table:homotopy-types2}
\end{table}

In the case $r=2$, $H_{3}( \vr{Q_{3}}{2})$ has rank one and is generated by a cross-polytopal element.
Even though $Q_4$ has only $2 \binom{4}{1}=8$ subcubes $Q_3$, we have $\rank H_{3}( \vr{Q_{4}}{2})=9>8$.
This indicates the appearance of a homology class $\alpha$ not generated by embedded homologies of $Q_3$ subcubes; see Section~\ref{sec:geometric} for a description of this ``geometric'' generator.
This new homology class contributes to the homology of higher-dimensional cubes in the same way as the homology described by Theorem~\ref{ThmMain4}.
We formalize this in the next subsection; see Theorem \ref{ThmMain5}.
Together, this cross-polytopal generator and this geometric generator $\alpha$ explain all of the homology when $r=2$:
\begin{enumerate}
\item[(a)] The cross-polytopal elements provide $2^{n-3} \binom{n}{3}$ independent generators for $H_{3}( \vr{Q_{n}}{2})$ (Theorem~\ref{ThmMain1}).
\item[(b)] The non-cross-polytopal element $\alpha$ provides $\sum_{i=4}^n 2^{i-4} \binom{i-1}{3}=\sum_{i=3}^{n-1} 2^{i-3} \binom{i}{3}$ more independent generators for $H_{3}( \vr{Q_{n}}{2})$ (Theorem~\ref{ThmMain5}).
\end{enumerate}
Thus, the combined lower bound says that the rank of $H_{3}( \vr{Q_{n}}{2})$ is at least
\begin{align*}
&2^{n-3} \binom{n}{3} + \sum_{i=3}^{n-1} 2^{i-3} \binom{i}{3} \\
=& \sum_{i=3}^{n} 2^{i-3} \binom{i}{3} \\
=& 2^{n-2}\binom{n}{3} - \sum_{i=1}^{n-2} 2^{i-1}\binom{i+1}{2} && \text{as explained in Appendix~\ref{app:2}}\\
=& \sum_{i=1}^{n-2} \left(2^{n-2}-2^{i-1}\right)\binom{i+1}{2} && \text{since }\binom{n}{3}=\sum_{i=1}^{n-2}\binom{i+1}{2} \\
=& \sum_{i=1}^{n-1} \left(2^{n-2}-2^{i-1}\right)\binom{i+1}{2} && \text{since }2^{n-2}-2^{(n-1)-1}=0 \\
=& \sum_{0 \leq j < i < n}(j+1)(2^{n-2}- 2^{i-1}) && \text{since }\binom{i+1}{2}=\sum_{j=0}^{i-1}(j+1) \\
\eqqcolon&\ c_n.
\end{align*}
Since $\vr{Q_n}{2} \simeq \bigvee_{c_n} S^3$ (Theorem~\ref{ThmAA}), this combined lower bound explains all of the homology when $r=2$.

\subsubsection{The case $\mathbf{r=3}$}

Using only the cross-polytopal generator for $\vr{Q_4}{3} \cong S^7$ and also the homotopy equivalence $\vr{Q_5}{3} \simeq \bigvee^{10} S^7 \bigvee S^4$, we obtain the following lower bounds bounds on homology when $r=3$:
\begin{enumerate}
\item[(a)] The cross-polytopal elements provide $2^{n-4} \binom{n}{4}$ independent generators for $H_{7}( \vr{Q_{n}}{3})$ (Theorem~\ref{ThmMain1}).
\item[(b)] The non-cross-polytopal four-dimensional homology element appearing at $n=5$ provides\\ $\sum_{i=5}^n 2^{i-5} \binom{i-1}{4} = \sum_{i=4}^{n-1} 2^{i-4} \binom{i}{4}$ independent generators for $H_{4}( \vr{Q_{n}}{3})$ (Theorem~\ref{ThmMain4}).
\end{enumerate}
The total lower bound equals the actual rank of all homology  of $\vr{Q_{n}}{3}$, due to Theorem~\ref{ThmZiqin} which says
\[\vr{Q_n}{3} \simeq \bigvee_{2^{n-4}\binom{n}{4}} S^7 \ \vee \bigvee_{\sum_{i=4}^{n-1} 2^{i-4} \binom{i}{4}} S^4.\]

\subsubsection{The case $\mathbf{r=4}$}
\label{SubSubsComparison:4}

Ziqin Feng at Auburn University has computed the homology of $\vr{Q_6}{4}$.
To do so, he used the Easley Cluster at Auburn University (a system for high-performance and parallel computing), about 180 GB of memory, and the Ripser software package~\cite{bauer2021ripser}.
His computations show that
\[
H_q(\vr{Q_6}{4};\Z/2) \cong \begin{cases}
0 & \text{if }1\le q\le 6\\
(\Z/2)^{239} & \text{if }q=7\\
0 & \text{if }8\le q\le 14\\
(\Z/2)^{14} & \text{if }q=15.
\end{cases}
\]
This computation is shown in the first $r=4$ column in Table~\ref{table:homotopy-types2}, and the consequences implied by this computation and by our results are shown in the second $r=4$ column in that table.

\subsubsection{The case of general $\mathbf{r}$}

In Section~\ref{sec:geometric} we prove that in each column of Tables~\ref{table:homotopy-types} or~\ref{table:homotopy-types2} with $r \geq 2$, a new homology generator appears in $\vr{Q_n}{r}$ with $n\ge r+1$, i.e., below the diagonal entry $n=r+1$ where the cross-polytopal generator appears.
Examples of these new homology generators are in the items (b) above for $r=2$ and $3$.

\subsection{Propagation of non-initial homology}

For positive integers $m < n$ let 
\[
\Psi_{m,n} \colon \coprod_{2^{n-m} \binom{n}{m}} Q_m \to Q_n
\]
denote the natural inclusion of all the $Q_m$ subcubes of $Q_n$.
Given a positive integer $q$, let the homomorphism
\[
(\Psi_{m,n})_* \colon \bigoplus_{{2^{n-m} \binom{n}{m}}}
H_q(\vr{Q_m}{r}) \to H_q(\vr{Q_n}{r})
\]
be the induced map on homology.

Our previous results have provided lower bounds on the rank of maps $(\Psi_{p,n})_*$, where $p$ is the smallest parameter for which $H_q(\vr{Q_p}{r})$ is non-trivial.
Our next result explains how an analogous result also holds for other maps $(\Psi_{m,n})_*$.

\begin{theorem}
 \label{ThmMain5}
 Let $q\ge 1$.
 Let 
 \[
 \rr_m = \rank \Big(H_q(\vr{Q_m}{r}) \ \big/ \ \image (\Psi_{m-1,m})_*\Big).
 \]
 Then for $n \geq m \geq p$, 
 \[
\rank \Big(H_q(\vr{Q_n}{r}) \ \big/ \ \image (\Psi_{m-1,n})_*\Big)
\geq 
\sum_{i=m}^n 2^{i-m} \binom{i-1}{m-1}
\cdot \rr_m.
\]
\end{theorem}
Note that $(\Psi_{p-1,p})_*=0$ by the definition of $p$, and so we recover Theorem~\ref{ThmMain4} by setting $m=p$ in Theorem~\ref{ThmMain5}.

\begin{proof}
The proof proceeds by induction on $m$.
We will actually prove 
\[
\rank(\rho_{m,n})
\geq 
\sum_{i=m}^n 2^{i-m} \binom{i-1}{m-1}
\cdot \rr_m,
\quad
\text{where}
\quad
\rho_{m,n}=\image (\Psi_{m,n})_* \ \big/ \ \image (\Psi_{m-1,n})_*
\]
The base case of the induction at $m=p$ is Theorem~\ref{ThmMain4}, since in this case $\image (\Psi_{p-1,p})_*=0$ as $H_q(\vr{Q_{p-1}}{r})=0$.
It remains to show the inductive step.
Our proof is essentially the same as the proof of Theorem~\ref{ThmMain4} applied to $\rho_{m,n}$ instead of to $H_q(\vr{Q_n}{r})$.

The cube $Q_n$ consists of two disjoint copies of $Q_{n-1}$; see Figure~\ref{Fig5}:
\begin{itemize}
    \item the rear one with the last coordinate $0$, denoted by $Q_{n-1}^0$, and
    \item the front one with the last coordinate $1$, denoted by $Q_{n-1}^1$.
\end{itemize}
We partition the $Q_m$ subcubes of $Q_n$ into three classes:
\begin{itemize}
    \item The ones contained in the rear $Q_{n-1}^0$ where vertices have last coordinate $0$, denoted by $R_n$.
    \item The ones contained in the front $Q_{n-1}^1$ where vertices have last coordinate $1$, denoted by $F_n$.
    \item The ones contained in the middle passage between them, denoted by $M_n$.
    Each such $Q_m$ in $M_n$ is of the form $D \times \{0,1\}$, where $D \subseteq Q_{n-1}$ is a copy of $Q_{m-1}$.
\end{itemize}

We will prove that the following $Q_m$ subcubes of $Q_n$ induce independent homology in $\rho_{m,n}$: 
the elements of $M_n$ (dashed cubes in Figure~\ref{Fig5}) and the elements of $F_n$ that have inductively been shown to include independent embeddings in $H_{q}(\vr{Q^1_{n-1}}{r}) /\image (\Psi_{m-1,n-1})_*$ (bold cubes in Figure~\ref{Fig5}).
The base case of the inductive process is Theorem~\ref{ThmMain4}.

The cardinality of $M_n$ is $2^{(n-1)-(m-1)} \binom{n-1}{m-1} = 2^{n-m} \binom{n-1}{m-1}$, which is the number of $Q_{m-1}$ subcubes in $Q_{n-1}^0$.
Each such subcube has the last coordinate constantly $0$.
Taking a union with a copy of the same $Q_{m-1}$ subcube with the last coordinates changed to $1$, we obtain a $Q_m$ subcube in $M_n$.
It is apparent that all elements of $M_n$ arise this way.
Let us enumerate the elements of $M_n$ as $Q_{m,j}^M$ with $1\le j \le 2^{n-m} \binom{n-1}{m-1}$.
For each such $j$ let $\{a_{i,j}\mid 1\le i \le \rr_m\}$ denote a largest linearly independent collection in $H_{q}( \vr{Q_{m,j}^M}{r})/\image (\Psi_{m-1,m})_*$.

The cardinality of the copies of $Q_m$ in $F_n$ that have inductively been shown to include independent embeddings in $H_{q}(\vr{Q^1_{n-1}}{r})/\image(\Psi_{m-1,n-1})_*$ equals $\sum_{i=m}^{n-1} 2^{i-m} \binom{i-1}{m-1}$, by inductive assumption.
Let us enumerate them by $Q_{m,j}^F$ with $1\le j\le \sum_{i=m}^{n-1} 2^{i-m} \binom{i-1}{m-1}$.
For each such $j$ let $\{b_{i,j} \mid 1\le i\le \rr_m\}$ denote a largest linearly independent collection in $H_{q}( \vr{Q_{m,j}^F}{r})/\image (\Psi_{m-1,m})_*$.

Assuming the equality 
\begin{equation}
\label{Eq3b}
\sum_{i,j}\lambda_{i,j} \cdot  a_{i,j} + \sum_{i,j}\mu_{i,j} \cdot  b_{i,j}=0    
\end{equation}
in $\rho_{m,n}$ for some coefficients $\lambda_{i,j}$, $\mu_{i,j}$, we claim that all coefficients equal zero.
This will prove the theorem as the number of involved terms equals $\sum_{i=m}^n 2^{i-m} \binom{i-1}{m-1} \cdot \rr_m$.

We will first prove that the coefficients $\lambda_{i,j}$ are all zero.
Fix some $1\le j \le 2^{n-m} \binom{n-1}{m-1}$, and let $D$ be the copy of $Q_{m-1}$ in $Q_{n-1}^0$ so that $C\coloneqq D \times \{0,1\}$ is equal to $Q_{m,j}^M$.
Let $f$ be any concentration $Q_n \to C:=D \times \{0,1\}$.
(For example, in Figure~\ref{Fig4} one can visualize $D$ as the solid round vertex, and $C$ as the edge between the two solid vertices.)
By Proposition~\ref{PropConcentration}:
\begin{itemize}
    \item $f$ maps any $Q_m$ subcube of $Q_n$ that contains $D$ bijectively onto $C=Q_{m,j}^M$.
    All such subcubes except for $C$ are contained in $R_n$.
    \item $f$ maps all of the other $Q_m$ subcubes of $Q_n$ to lower-dimensional subcubes, and hence the map 
\[ H_q (\vr{Q_m}{r})\ \big/ \ \image (\Psi_{m-1,m})_* \to H_q(\vr{C}{r})\ \big/ \ \image (\Psi_{m-1,m})_* \]
induced by $f|_C$ is trivial.
\end{itemize}
These two observations imply that applying the induced map $f_*$ on homology to Equation~\eqref{Eq3b}, we obtain 
$
\sum_{i}\lambda_{i,j} \cdot  a_{i,j}=0.
$
By the choice of $\{a_{i,j}\}_i$ as an independent collection of homology classes, we obtain $\lambda_{i,j}=0$ for all $i$.
Since this can be done for any $1\le j \le 2^{n-p} \binom{n-1}{p-1}$, we have $\lambda_{i,j}=0$ for all $i$ and $j$.

We have thus reduced Equation~\eqref{Eq3b} to 
$
\sum_{i,j}\mu_{i,j} \cdot  b_{i,j}=0.
$
Let $\pi_S \colon Q_n \to Q_{n-1}$ be the projection that forgets the last coordinate of each vector (explicitly, $S=[n-1]\subseteq [n]$).
Note that the restrictions of $\pi_S$ to $Q_n^0$ and to $Q_n^1$ are bijections.
Hence, after applying the induced map $(\pi_S)_*$ on homology, the inductive definition of the $b_{i,j}$ as being independent in $H_{q}(\vr{Q^1_{n-1}}{r})/\image (\Psi_{m-1,n-1})_*$ implies that $\mu_{i,j}=0$ for all $i$ and $j$.
\end{proof}

\section{Geometric generators}
\label{sec:geometric}

For $r\ge 0$, let $f_n\colon \vr{Q_n}{r}\to[0,1]^n$ be the map defined by sending each vertex of $Q_n$ to the corresponding point in $\{0,1\}^n\subseteq [0,1]^n$, and then by extending linearly to simplices.
Let $n(r)\in\{0,1,\ldots\}$ be the smallest integer $n$ such that $f_n\colon \vr{Q_n}{r}\to[0,1]^n$ is not surjective.
The fact that this is well-defined follows from Lemma~\ref{LemGeomGen}, which proves that $n(r) \leq 2r+1$.

The main theorem in this section is the following.

\begin{theorem}
\label{ThmGeom}
For all $r\ge 2$,
\begin{description}
    \item[i] There exists some $m \leq n(r)$ such that $\pi_{m-1}(\vr{Q_m}{r})\neq 0$;
    \item[ii] There exists some $k \leq m$ such that $H_{k-1}(\vr{Q_m}{r})\neq 0$;
    \item[iii] Not all of the above non-trivial homotopy group (resp.\ homology group) is generated by the initial cross-polytopal spheres, i.e., the image of the induced map $(\Psi_{r+1,m})_*$ of Vietoris-Rips complexes $\vr{Q_{r+1}}{r} \hookrightarrow \vr{Q_m}{r}$ at scale $r$ is not all of $\pi_{m-1}(\vr{Q_m}{r})$ (resp.\ $H_{k-1}(\vr{Q_m}{r})$).
\end{description}
\end{theorem}
\noindent Statement \textbf{ii} above is true for homology taken with any choice of coefficients.

An important consequence of this theorem is the following.
Together, statements \textbf{i}--\textbf{iii} imply that for each $r\ge 2$, there is a new topological feature in $\vr{Q_m}{r}$ that is not induced from an inclusion $\vr{Q_{r+1}}{r}\hookrightarrow\vr{Q_m}{r}$.
In Table~\ref{table:homotopy-types2}, these appear as the new blue $S^3$ feature in $\vr{Q_4}{2}\simeq \vee^9 S^3$, and as the new blue $S^4$ feature in $\vr{Q_5}{3}\simeq \vee^{10} S^7 \vee S^4$.

In the following example, when $r=2$, we see that we can take $k=m=n(r)$.
However, we do not know if this is the case in general.

\begin{example}
\label{ex:geom-r=2}
Fix $r=2$.
Note that $n(2)=4$ is the smallest integer $n$ such that $f_n\colon \vr{Q_n}{2}\to[0,1]^n$ is not surjective.
We will use this to show $\pi_3(\vr{Q_4}{2})\neq 0$.
The following five tetrahedra form a triangulation of $[0,1]^3$ (see Section~5.3 and in particular Figure~4 of ~\cite{VargasRosarioThesis}).
\begin{align*}
\tau_1 &= \{(0,1,1),(1,1,0),(1,0,1),(0,0,0)\} \\
\tau_2 &= \{(0,0,0),(1,0,0),(1,1,0),(1,0,1)\} \\
\tau_3 &= \{(0,1,1),(1,1,0),(0,0,0),(0,1,0)\} \\
\tau_4 &= \{(1,1,1),(0,1,1),(1,1,0),(1,0,1)\} \\
\tau_5 &= \{(0,0,0),(0,1,1),(0,0,1),(1,0,1)\}
\end{align*}
Furthermore, each tetrahedron $\tau_i$ has diameter at most $2$.
Since $[0,1]^3=\cup_{i=1}^5 |\tau_i|$, and since each $\tau_i$ is a simplex in $\vr{Q_3}{2}$, we define a surjective map $h\colon [0,1]^3 \to \vr{Q_3}{2}$ by letting $h(x)=\sum_{v\in \sigma(x)}\lambda_v v$, where $\sigma(x)$ is the unique simplex in the triangulation $[0,1]^3=\cup_{i=1}^5 |\tau_i|$ that contains $x$ in its interior.\footnote{By convention, the interior of a vertex is that vertex.}
We note that the map $h$ is continuous.
Note that $\partial([0,1]^4)$ is composed of $8$ faces, where each face is a 3-dimensional cube $[0,1]^3$.
By piecing together $8$ copies of the map $h$ in a continuous way, we obtain a map $g\colon \partial([0,1]^4) \to \vr{Q_4}{2}$.
The map $f_4\colon \vr{Q_4}{2}\to[0,1]^4$ is not surjective, although its image does contain $\partial([0,1]^4)$, which follows since $f_3\colon \vr{Q_3}{2}\to[0,1]^3$ is surjective.
Therefore, there exists a point $v\in \interior([0,1]^4)\setminus \image(f_4)$.
We thus obtain a composition
\[ \partial([0,1]^4) \xrightarrow{g} \vr{Q_4}{2} \xrightarrow{f_4} [0,1]^4\setminus\{v\} \xrightarrow{\pi_v} \partial([0,1]^4), \]
where $\pi_v \colon [0,1]^4\setminus\{v\} \to \partial([0,1]^4)$ is the radial projection away from the point $v\in \interior([0,1]^4)$.
Note that the composition $\pi_v \circ f_4 \circ g$ is equal to the identity map on $\partial([0,1]^4)$, and that we have a homeomorphism $ \partial([0,1]^4) \cong S^3$.
Since this map $\pi_v \circ f_4 \circ g$ factors through $\vr{Q_4}{2}$, it follows that $\pi_3(\vr{Q_4}{2})\neq 0$.
The proof of Theorem~\ref{ThmGeom} follows this strategy.
\end{example}

\begin{example}
Fix $r=3$.
By~\cite{feng2023homotopy} we have $\vr{Q_5}{3} \simeq \bigvee^{10} S^7 \vee S^4$.
Theorem~\ref{ThmGeom} is satisfied with $m=k=5$.
\end{example}

\begin{nonexample}
\label{nonexample-feng}
Fix $r=4$.
By Ziqin Feng's homology computations in Section~\ref{SubSubsComparison:4} we have $H_7(\vr{Q_6}{4};\Z/2)\cong (\Z/2)^{239}$, $H_{15}(\vr{Q_6}{4};\Z/2)\cong (\Z/2)^{14}$, and $H_{q}(\vr{Q_6}{4};\Z/2)=0$ for $1\le q\le 6$ and $8 \le q\le 14$.  
Since the dimensions $7$ and $15$ of nontrivial homology are not smaller than $6$, neither of these nontrivial homology groups fits the description in Theorem~\ref{ThmGeom}(\textbf{ii}).
This shows that $m>6$ when $r=4$.
In other words, Theorem~\ref{ThmGeom} guarantees that a new topological feature not yet present in the $r=4$ row of Table~\ref{table:homotopy-types2} will appear in some $\vr{Q_m}{4}$ with $m>6$.
\end{nonexample}

We now build up towards the proof of Theorem~\ref{ThmGeom}.
We will use the following notation:
\begin{itemize}
    \item $C_n = [0,1]^n$ denotes the $n$-dimensional cube.
    \item $k$-skeleton: $C_n^{(k)}= \cup \{k\text{-dimensional subcubes of } C_n \}= \cup \{\conv Q_k \mid Q_k \leq Q_n\}$.
    \item $\partial C_n = C_n^{(n-1)} \cong S^{n-1}$.
\end{itemize}

We will use the following lemma in the proof of Theorem~\ref{ThmGeom}.

\begin{lemma} 
    \label{LemGeomGen}
    For each $r\ge 2$ the following hold:
    \begin{enumerate}
        \item $n(r) \leq 2r+1$.
        \item The image of $f_{n(r)}$ contains $\partial C_{n(r)-1}$.
        \item For each $r\ge 2$ and for each $n$, $\vr{Q_n}{r}$ is connected and simply connected.
    \end{enumerate}
\end{lemma}

\begin{proof}
    (1) Choose a simplex $\sigma \in \vr{Q_{n(r)-1}}{r}$ whose image via $f_{n(r)-1}$ contains the center $z=(\frac{1}{2}, \frac{1}{2}, \ldots, \frac{1}{2})$ of the cube. Without loss of generality (due to the symmetry) we may assume $\sigma$ contains the origin $(0, 0, \ldots, 0)$.
    As $||z||_{\ell_1}= \frac{n(r)-1}{2}$, $\sigma$ must contain a vertex of $\ell_1$-norm at least $\frac{n(r)-1}{2}$.
    On the other hand, the $\ell_1$-norm of each vertex of $\sigma$ is at most $r$, due to the inclusion of the origin in $\sigma$. Thus $\frac{n(r)-1}{2} \leq r$, which implies (1).  
    
    (2) Holds since $f_{n(r)-1}$ is not surjective. 

    (3) Choose $r>0$.
    The complex $\vr{Q_n}{r}$ is obviously connected.
    Let $\alpha$ be a based simplicial loop in $\vr{Q_n}{r}$ and let $[v,w]$ be an edge, which is a part of $\alpha$.  Setting $v=v_0$ and $w=v_r$, we can replace $[v,w]$ by a homotopic (rel \{v,w\}) concatenation of edges $[v_0,v_1] * [v_1, v_2]* \ldots* [v_{r-1},v_r]$, whose pairwise distances are at most $1$. In particular, since $v$ and $w$ differ in at most $r$ coordinates, we can choose $v_i$ inductively so that $v_i$ differs from $v_{i+1}$ in at most one coordinate. Thus $\diam\{v_0, v_1, \ldots, v_r\} \leq r$ which means that the defined vertices $v_i$ form a simplex contained in $\vr{Q_n}{r}$. As any simplex is contractible, $[v,w]$ is homotopic (rel \{v,w\}) to the concatenation of edges $[v_0,v_1] * [v_1, v_2]* \ldots* [v_{r-1},v_r]$. 
    
    Replacing each edge of $\alpha$ in this manner we obtain a based homotopic simplicial loop $\beta$, such that the endpoints of all the edges are at distance at most $1$.
    The loop $\beta$ is thus contained $\vr{Q_n}{2}$, which is simply connected by~\cite{adamaszek2022vietoris}, and contained in $\vr{Q_n}{r}$. As a result, $\beta$ and $\alpha$ are contractible loops. As $\alpha$ was arbitrary, this means $\vr{Q_n}{r}$ is simply connected. 
\end{proof}

\begin{proof}[Proof of Theorem~\ref{ThmGeom}]

\textbf{i}: Fix $r\ge 2$.
In order to show \textbf{i}, it is equivalent to assume $\pi_{n-1}(\vr{Q_n}{r})$ is trivial for all $n \in \{1,2, \ldots, n(r)-1\}$, and then prove that $\pi_{N-1}(\vr{Q_{N}}{r})\neq 0$ for $N\coloneqq n(r)$.
This is what we will do.

First, we define $\f \colon \partial C_{N} \to \vr{Q_{N}}{r}$ by induction on the skeleta of $C_{N}$.
For the base case, the 0-skeleton, define $\f|_{C_{N}^{(0)}} $ as the identity on $Q_{N}$, mapping a point of $Q_{N}$ to the corresponding vertex in $\vr{Q_{N}}{r}$.
Now, assume that 
$$
\f \colon C_{N}^{(j)} \to \vr{Q_{N}}{r}
$$
has been defined for some $j \in \{0,1,\ldots, N-2\}$ in a \textbf{subcube-preserving} manner,. i.e., 
\begin{equation}
\label{eq:subcube-preserving}
\forall i \leq j, \ \forall Q_i \leq Q_{N}: \ \f(\conv Q_i) \subseteq \vr{Q_i}{r}.
\end{equation}
Before defining $\f$ on all of $C_N^{(j+1)}$, we first explain how to define $\f$ on a single $(j+1)$-dimensional cube.
Fix $Q_{j+1} < Q_{N}$, and let $C_{j+1} = \conv Q_{j+1}$. 
Define $\f|_{C_{j+1}} \colon C_{j+1} \to \vr{Q_{j+1}}{r}$ as follows:
\begin{itemize}
    \item By \eqref{eq:subcube-preserving} above we have 
    $$
    \f({\partial C_{j+1}}) \subseteq \cup_{Q_j \leq Q_{j+1}} \vr{Q_j}{r} \leq  \vr{Q_{j+1}}{r}.
    $$
    \item By the assumption at the beginning of the proof, $\f|_{\partial C_{j+1}} \colon \partial C_{j+1} \to \vr{Q_{j+1}}{r}$ is contractible and can thus be extended over $C_{j+1}$.
    In particular,  the subcube-preserving condition $\f({ C_{j+1}}) \subseteq   \vr{Q_{j+1}}{r}$ holds.
\end{itemize}
Defining $\f$ on $\conv Q_j$ for each $ Q_j \leq Q_{j+1}$ we obtain a continuous subcube-preserving map $\f$ defined on $\partial C_{N}^{(j+1)}$.
This concludes the inductive step, and thus we obtain a subcube-preserving map 
$$
\f \colon \partial C_{N} \to \vr{Q_{N}}{r}.
$$

Next, we next show that $\f$ is not contractible.
Choose $z\in C_{N} \setminus \partial C_{N}$ such that $z$ is not contained in the image of $f_{N}$.
Let $\nu \colon C_{N} \setminus \{z\} \to \partial C_{N}$ be the radial projection map, which is a retraction.
Define 
$\psi = \nu \circ f_{N} \circ \f \colon \partial C_{N} \to \partial C_{N}$
as the composition of maps
$$
\partial C_{N} \stackrel{\f}{\to} \vr{Q_{N}}{r}\stackrel{f_{N}}{\to} C_{N} \setminus \{z\} \stackrel{\nu}{\to} \partial C_{N},
$$
and note it is a map between topological $(N-1)$-spheres.
Observe that $\psi$ is subcube-preserving, i.e., $\forall Q_j < Q_{N}: \psi(\conv Q_j) \subseteq \conv Q_j$.
Map $\psi \colon \partial C_{N} \to \partial C_{N}$ is homotopic to the identity, as is demonstrated by the linear homotopy 
$$
H \colon \partial C_{N} \times I \to \partial C_{N}, \qquad H(x,t) = (1-t)\f(x) + t x,
$$
which is well-defined by the subcube-preserving property.
As $\psi$ is homotopically non-trivial, so is $\f$.
Since $\vr{Q_{N}}{r}$ is path connected, this implies $\pi_{N} (\vr{Q_{N}}{r})$ is non-trivial regardless of which basepoint is used, giving \textbf{i}.

\textbf{ii}: By the Hurewicz theorem and Lemma \ref{LemGeomGen}(3), the first non-trivial homotopy group of $\vr{Q_m}{r}$ is isomorphic to the corresponding homology group with integer coefficients in the same dimension.
By \textbf{i}, the mentioned dimension is at most $m-1$.

\textbf{iii}: Let $m \in \{1,2, n(r)\}$ be the parameter from the proof of \textbf{i} that satisfies $\pi_{m-1}\vr{Q_{m_r}}{r}\neq 0$.
For $r>2$ we claim that $m-1 < 2^r-1$.
Indeed, $m-1 \le n(r)-1 \leq 2r < 2^r-1$ by Lemma \ref{LemGeomGen}(2). 
So for $r>2$, the dimension ($m-1$ or lower) of the homotopy and homology groups in \textbf{i} and \textbf{ii} is lower than the dimension $2^r-1$ of the $\Psi_{r+1}^{m_r}$ induced invariants, giving \textbf{iii}.
Finally, in the case $r=2$, we have $m=2$ and $\vr{Q_4}{2} \simeq \vee^9 S^3$ by~\cite{adamaszek2022vietoris}.
Since $Q_4$ contains only $8$ copies of $Q_3$, the image of $\Psi_3^4$ induced map on $\pi_3$ is of rank at most $8$; thus the claim \textbf{iii} follows also for $r=2$.
\end{proof}

\section{Conclusion and open questions}
\label{sec:conclusion}

We conclude with a description of some open questions.
We remind the reader of questions from~\cite{adamaszek2022vietoris}, which ask if $\vr{Q_n}{r}$ is always a wedge of spheres, what the homology groups and homotopy types of $\vr{Q_n}{r}$ are for $3\le r\le n-2$, and if $\vr{Q_n}{r}$ collapses to its $(2^r-1)$-skeleton (which would imply that the homology groups $H_q(\vr{Q_n}{r})$ are zero for $q\geq 2^r$).
Below we pose some further questions.

The first four questions are related to the geometric generators in Section~\ref{sec:geometric}.
Understanding the answers to any of them would provide further information about parameters of Theorem~\ref{ThmGeom}. 

\begin{question}
Recall that $n(r)\in\{0,1,\ldots\}$ is the smallest integer $n$ such that $f_n\colon \vr{Q_n}{r}\to[0,1]^n$ is not surjective.
What are the values of $n(r)$ as a function of $r$?
\end{question}

\begin{question}
If $f_n\colon \vr{Q_n}{r}\to[0,1]^n$ is not surjective, then is it necessarily the case that the center $(\frac{1}{2}, \frac{1}{2}, \ldots, \frac{1}{2})$ is not in the image of $f_n$?
\end{question}

\begin{question}
If $f_n\colon \vr{Q_n}{r}\to[0,1]^n$ is surjective, then does their exist a triangulation of $[0,1]^n$ by simplices of diameter at most $r$?
\end{question}

\begin{question}
The following question is based on a StackExchange post~\cite{StackExchangeBalanced}.
A subset $B\subseteq\{0,1\}^n$ is \emph{balanced} if $\frac{1}{|B|}\sum_{b\in B} b = (\frac{1}{2},\ldots,\frac{1}{2})\in\R^n$.
For example, the tetrahedron $\tau_1$ in Example~\ref{ex:geom-r=2} is a set of four vertices that forms a balanced subset.
If $(\frac{1}{2}, \frac{1}{2}, \ldots, \frac{1}{2})$ is in the image of $f_n\colon \vr{Q_n}{r}\to[0,1]^n$, then does there necessarily exist a balanced subset of $\{0,1\}^n$ of diameter at most $r$?
One of the reasons we ask this question is that the answers to the StackExchange post~\cite{StackExchangeBalanced} place constraints on the smallest diameter for a balanced subset of the $n$-dimensional cube.
\end{question}

The remaining questions are more general.

\begin{question}
In Section~\ref{sec:maximal} we described cross-polytopal homology generators.
In Section~\ref{sec:geometric} we described geometric homology generators.
In Non-Example~\ref{nonexample-feng} we described homology generators, due to computations by Ziqin Feng, that are neither cross-polytopal in the sense of Section~\ref{sec:maximal} (arising from an isometric embedding $Q_{r+1}\hookrightarrow Q_n$) nor geometric in the sense of Section~\ref{sec:geometric}.
What other types of homology generators are there for $H_q(\vr{Q_n}{r})$?
\end{question}

\begin{question}
Our main results show how the homology (and persistent homology) of $\vr{Q_p}{r}$ for $1\le p \le m$ place lower bounds on the Betti numbers of $\vr{Q_n}{r}$ for all $n\ge m$.
For every $r\ge 1$, is there some integer $m(r)$ such that our induced lower bounds are tight for all $n\ge m(r)$ and for all homology dimensions?
\end{question}

\begin{question}
The group of symmetries of the $n$-dimensional cube is the hyperoctahedral group.
How does this group act on the homology $H_q(\vr{Q_n}{r})$?
\end{question}



\begin{question}
What homology propogation results can be proven for \v{C}ech complexes of hypercube graphs, as studied in~\cite{adams2022v}?
\end{question}

\section*{Acknowledgements}

We would like to thank the Institute of Science and Technology Austria (ISTA) for hosting research visits, and
we would like to thank John Bush, Ziqin Feng, Michael Moy, Samir Shukla, Anurag Singh, Daniel Vargas-Rosario, and Hubert Wagner for helpful conversations.
The second author was supported by Slovenian Research Agency grants No. N1-0114, J1-4001, and P1-0292.
We acknowledge the Auburn University Easley Cluster for support of this work, via computations carried out by Ziqin Feng.




\bibliographystyle{plain}
\bibliography{VRHypercube-LowerBoundHomology.bib}

\appendix

\section{Proofs of numerical identities}

This appendix contains two short proofs of numerical identities, both of which are used in Section~\ref{SubsComparison} in the cases $r=1$ and $r=2$.

\subsection{First proof}
\label{app:1}

Let $S_n=\sum_{i=2}^n (i-1) 2^{i-2}$; we claim $S_n = n2^{n-1}-2^n+1$.
Indeed, note that
\begin{align*}
S_n=2 S_n - S_n&=0 &&+2 &&+2\cdot2^2 &&+3\cdot 2^3 &&+\ \ \ \ \ \ldots &&+(n-2)\cdot 2^{n-2}&&+(n-1)\cdot 2^{n-1}\\
&\ \ \ -1 &&-2\cdot 2 &&-3\cdot2^2 &&-4\cdot 2^3 &&-\ \ \ \ \ \ldots &&-(n-1)\cdot 2^{n-2}\\ &=-1 &&-2 &&-2^2 &&-2^3 &&-\ \ \ \ \ \ldots &&-2^{n-2}&&+(n-1)\cdot 2^{n-1},
\end{align*}
which gives $S_n = -(2^{n-1}-1)+(n-1)2^{n-1} = n2^{n-1}-2^n+1$.

\subsection{Second proof}
\label{app:2}

We prove that $\sum_{i=3}^{n} 2^{i-3} \binom{i}{3}
= 2^{n-2}\binom{n}{3} - \sum_{i=1}^{n-2} 2^{i-1}\binom{i+1}{2}$ by induction on $n$.
For the base case $n=3$, note that both sides equal $1$.
For the inductive step, note that if we assume the formula is true for $n-1$, then as desired we get
\begin{align*}
\sum_{i=3}^{n} 2^{i-3} \binom{i}{3}
&= 2^{n-3} \binom{n}{3} + \left(2^{n-3}\binom{n-1}{3} - \sum_{i=1}^{n-3} 2^{i-1} \binom{i+1}{2}\right) \\
&= 2^{n-3} \binom{n}{3} + 2^{n-3}\binom{n-1}{3}  + 2^{n-3}\binom{n-1}{2}- \sum_{i=1}^{n-2} 2^{i-1} \binom{i+1}{2} \\
&= 2^{n-3} \binom{n}{3} + 2^{n-3}\binom{n}{3}- \sum_{i=1}^{n-2} 2^{i-1} \binom{i+1}{2} \\
&= 2^{n-2}\binom{n}{3} - \sum_{i=1}^{n-2} 2^{i-1}\binom{i+1}{2}.
\end{align*}
\\\\

\end{document}